\documentclass[11pt]{article}
\usepackage{amsmath, amssymb, amsfonts, amsthm, authblk, color, graphicx, enumitem, mathrsfs, indentfirst}
\usepackage{hyperref, cleveref}
\hypersetup{colorlinks=true, linkcolor=blue, filecolor=magenta, urlcolor=cyan}
\usepackage[textwidth=7in,textheight=9in]{geometry}
\usepackage{float}

\allowdisplaybreaks

\newtheorem{theorem}{Theorem}

\newtheorem{lemma}{Lemma}[section]

\newtheorem{remark}{Remark}[section]

\newtheorem{corollary}{Corollary}[section]
\newtheorem{example}{Example}

\numberwithin{equation}{section}

\newcommand{\keywords}[1]{\small\textbf{\textit{Keywords---}}#1}

\title{Some new numerical schemes for delay sub-diffusion equations via the idea of symmetry}

\author[2]{Dakang Cen}
\author[1]{Caixia Ou\footnote{Corresponding author: oucaixiaywxk@163.com.}}
\author[3]{Seakweng Vong}
\author[4]{Zhibo Wang}
\affil[1]{College of Mathematics and Informatics, South China Agricultural University, Guangzhou 510642, China.}
\affil[2]{Department of Mathematics, Southern University of Science and Technology, Shenzhen, China.}
\affil[3]{Department of Mathematics, University of Macau, Macao, China.}
\affil[4]{School of Mathematics and Statistics, Guangdong University of Technology, Guangdong, Guangzhou 510006, China.}
\date{}
\begin{document}
\maketitle

\abstract{%This work introduces a symmetric fractional-order reduction (SFOR) method to construct robust numerical algorithms for fractional delayed sub-diffusion equations.
In this work, we propose several novel numerical schemes that synergistically integrates the $L1$ method and symmetric fractional-order reduction (SFOR) method, specifically designed for fractional sub-diffusion equations with time delay. Crucially, we rigorously establish an optimal convergence order of $2 - \alpha/2$ in temporal direction, which improves upon the classical rate of $2-\alpha$ under identical smoothness conditions. To further enhance computational efficiency, we develop a weighted alternating direction implicit (ADI) scheme for the two-dimensional case and analyze its coupling with compact spatial discretization---thereby achieving high-order accuracy in space. Moreover,  we introduce a new framework for analyzing the error of $L1$-type discretization on graded temporal meshes, which relaxes the mesh grading requirement and still guarantees optimal convergence rate at continuous time points. Some numerical experiments comprehensively validate both the theoretical accuracy and computational efficiency of the resulting schemes.}



\section{Introduction}
Assuming that $\alpha\in(0,1)$ and $\Omega\subset \mathbb{R}^{\mathrm{d}},~\mathrm{d}=1,2$, we consider the following sub-diffusion equation with time delay:
\begin{equation}\label{eq-gov}
\begin{cases}
  \partial_t^\alpha u - \Delta u + u_{-\tau} = f(\textbf{x},t), & (\textbf{x},t) \in \Omega \times (0,K\tau], \\
  u(\textbf{x},t) = a_0(\textbf{x},t), & (\textbf{x},t) \in \bar{\Omega}\times[-\tau,0], \\
  u(\textbf{x},t) = 0, & (\textbf{x},t) \in \partial \Omega \times (0,K\tau],
\end{cases}
\end{equation}
where $\partial\Omega$ is the bound of $\Omega$, $\bar{\Omega}=\Omega\cup\partial\Omega$, $\Delta$ is Laplace operator, $u_{-\tau}=u(t-\tau)$, $\tau>0$ is a constant time delay and the operator $\partial_t^\upsilon$ is known as the Caputo derivative of order $\upsilon$:
\[\partial_t^\upsilon u(t):= (\mathcal{I}^{n-\upsilon}u^{(n)}
)(t)~\mbox{for}~ t>0 ~\mbox{and}~ n-1<\upsilon< n,\]
in which $\mathcal{I}^{\beta}$ represents the Riemann-Liouville fractional integral of order $\beta$:
\[\mathcal{I}^{\beta}u(t):=
 \int_0^t \omega_\beta(t-s)u(s) ds~\mbox{with}~\omega_\beta(t)=\frac{t^{\beta-1}}{\Gamma(\beta)}, ~\beta>0.
\]

 In the real world, time delays are prevalent and have been incorporated into various mathematical models, including automatic control systems with feedback mechanisms \cite{Si-Ammour}, population dynamics \cite{KuangY,LenburyY} and so on. Moreover, time delay has also been utilized in modeling HIV infection of CD$4^+$T-cells to characterize the time between the infection of these cells and the emission of viral particles at the cellular level \cite{CulshawR,YanY}. And thus differential equations involving time delay have garnered much attention from an increasing number of scholars \cite{Bellen,LiDF,OuJCAM2025,Wille}.
In particular, fractional diffusion models with time delay have been constructed and studied both theoretically \cite{Morgado} and numerically \cite{Hendy1,Hendy,LiL}. However, the above mentioned work is primarily valid for sufficiently smooth solutions. Few studies have addressed this type of fractional model with non-smooth solutions, apart from those by references \cite{Delay03-Cen,Delay04-Cen,OuCNSNS,Delay02-Cen,Delay01-Tan}.

\begin{figure}[!htb]
	\setlength{\abovecaptionskip}{0.cm}
	\setlength{\belowcaptionskip}{-0.cm}
	\centering
    \includegraphics[width=8cm]{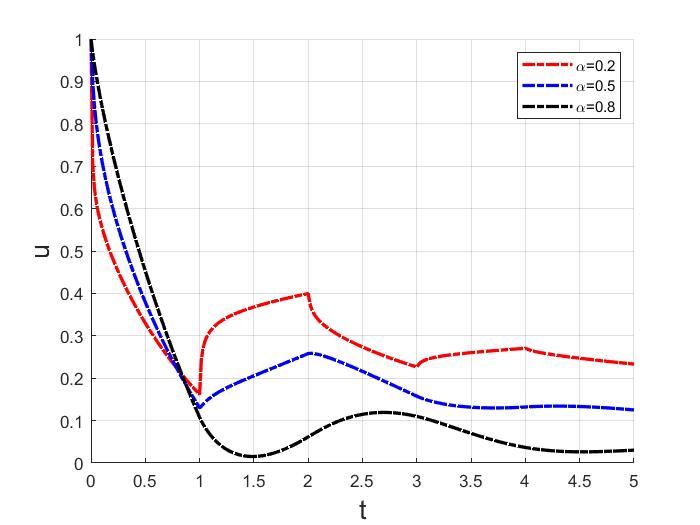}
	\caption{Multi-singularity problems arising from nonlocal delay models \cite{Delay03-Cen,Delay04-Cen}. }
    \label{fig1}
\end{figure}

Especially, Cen $et~al.$ \cite{Delay03-Cen} found that weak singularity near $t=0$ may propagate at integer multiples of the fixed delay in the delay fractional equations. Moreover, the solution exhibits increased regularity in each subsequent time interval beyond a singularity point compared to the preceding one, that is,
\begin{align}\label{regularity}
\Big|\frac{du(t)}{d t}\Big|\leq C\big(1+(t-k\tau)^{(k+1)\alpha-1}\big)
\end{align}
for $t\in(k\tau,(k+1)\tau],~k=0,1,2,\ldots,K-1.$
Subsequently, local error estimates for the delayed reaction-subdiffusion equation were established on uniform meshes using both the $L1$ method \cite{Delay05-Bu,Delay02-Cen} and Gr\"{u}nwald-Letnikov-type approximation \cite{Delay07-Bu}. For fractional differential equations (FDEs) without time delay, substantial progress has been made in designing nonuniform temporal meshes and fitted schemes to resolve initial singularities; representative advances include those reported in \cite{DengXY,LiaoH2018L1,LiaoH2021L2_1} and related works. Whereas the study on how to deal with the problems with multiple singularities effectively is still rare. To address this kind of problem, Tan $et~al.$ \cite{Delay01-Tan} proposed $L1$ method based on nonuniform meshes graded densely near all identified singular points (including $t=0$ and its integer-delay shifts) and coarsened elsewhere.
%Unfortunately, the strategy of mesh subdivision does not apply to the methods that have limitations on time steps, such as $L2$-$1_\sigma$ formula. This is because the step ratios of grid may not satisfy the condition in \cite{LiaoH2021L2_1}, and thus the great properties of discrete coefficients can not be guaranteed, such as monotonicity, boundedness, etc, which makes the convergence theory difficult.
Later, Cen $et~al.$ \cite{Delay04-Cen} employed the decomposition technique to isolate the non-smooth component and developed a fitted $L1$ scheme on uniform mesh for delay sub-diffusion equations. However, this approach is inherently limited to single-singularity problems and further requires stringent compatibility conditions between the initial values and the source term $f$. Therefore, to achieve high-order temporal accuracy for multi-singularity problems in delay fractional models, two corrected $L$-type methods---the $L1$ and $L2$-$1_\theta$ schemes---were recently introduced in \cite{Delay03-Cen}, both relying on prior knowledge of the asymptotic behavior of the exact solution near each singular point. Yet, except for a few analytically tractable cases, determining such asymptotics remains highly challenging, as exact solutions are generally unavailable for most delay FDEs. A systematic comparison of these existing numerical schemes for delay fractional sub-diffusion equations is summarized in Table \ref{REF-method}.

\begin{table}[!ht]\label{REF-method}
\caption{Numerical schemes for delay fractional differential equations, $\alpha\in(0,1)$.}
\renewcommand{\arraystretch}{1}
\def\temptablewidth{0.8\textwidth}
\begin{center}
\begin{tabular*}{\temptablewidth}{@{\extracolsep{\fill}}lccc}\hline
 scheme       & temporal rate              &time grid    & type of estimation        \\ \hline
$L1$ \cite{Delay05-Bu,Delay02-Cen}          & $\min\{(k+1)\alpha,1\}$       &uniform  & pointwise/local \\ \hline
G-L \cite{Delay07-Bu}           &$\alpha$ or 1              &uniform  & pointwise/local \\ \hline
Fitted $L1$ \cite{Delay04-Cen}          & $\min\{2\alpha,2-\alpha\}$       &uniform   & global                       \\ \hline
Corrected $L1$ \cite{Delay03-Cen}          & $\min\{3\alpha,2-\alpha\}$       &uniform   & global                       \\ \hline
Corrected $L2$-$1_\theta$ \cite{Delay03-Cen}          & $\min\{3\alpha,2\}$       &uniform   & global                       \\ \hline
$L1$ \cite{OuCNSNS,Delay06-Li,Delay01-Tan}            & $2-\alpha$    &nonuniform    & global           \\ \hline
SYM $L1$  & $2-\alpha/2$ & nonuniform & global \& local \\ \hline
SYM $L1$-ADI  & $\frac{3}{2}\alpha$ & nonuniform & global  \\ \hline
\end{tabular*}
\end{center}
\end{table}

However, as shown in Table \ref{REF-method}, the classical $L1$ method for discretizing the Caputo fractional derivative in time-delayed models yields only an optimal temporal convergence order of $2-\alpha$. To bridge this theoretical gap, we apply the $L1$ method not directly to the original delayed fractional sub-diffusion equation, but to an equivalent coupled system derived via the symmetric fractional-order reduction (SFOR) technique, which was first proposed in \cite{LyuP2022SFOR}. Specifically, SFOR introduces an auxiliary function $v=\partial_t^{\frac{\alpha}{2}}u$ to reformulate the original single-equation problem into a coupled system of lower-order fractional equations---constituting a non-fixed-order reduction strategy.
We rigorously establish that the resulting nonuniform SYM $L1$ scheme achieves the improved temporal convergence rate $\mathcal{O}\big(N^{-\min\{\gamma\frac{\alpha}{2},\, 2-\alpha/2\}}\big)$, as stated in Theorem \ref{thm-conv-1}. Furthermore, we extend this scheme to an alternating direction implicit (ADI) framework---the SYM $L1$-ADI method---which significantly reduces both memory requirements and computational complexity while preserving the  convergence order $\mathcal{O}\big(N^{-\min\{\gamma\frac{\alpha}{2},\, \frac{3}{2}\alpha\}}\big)$.

Nevertheless, global error analysis imposes stringent constraints on the time discretization, particularly requiring sufficiently small step sizes to guarantee convergence; this constraint becomes especially restrictive when the fractional order $\alpha$ is small. In contrast, local error analysis is particularly well-suited for cases with small $\alpha$ \cite{Kopteva2019}, as it naturally accommodates the initial singularity of the solution and mitigates excessive accumulation of local truncation error under long-term, slowly decaying integration dynamics. What's more, given that the existing analyses presented in \cite{Delay07-Bu,Delay05-Bu,Delay02-Cen} for delayed fractional models are limited to first‑order accuracy, it is highly imperative to enhance the precision of the numerical schemes. Consequently, we adopt a new local error analysis framework for the given model, thereby relaxing the rigid time-step dependence in global convergence theory and establishing a more efficient and accurate theoretical foundation for numerical simulations in the small-$\alpha$ regime.
This framework is simple and applicable to both finite difference and finite element spatial discretizations. We use more intuitive integral representations of the temporal truncation errors; see Lemma \ref{case-2}. This new proof relies on a simple barrier function, and may be of independent interest. Once error bounds on graded meshes are established for a paradigm problem without spatial derivatives, they seamlessly extend
to finite difference and finite element spatial discretizations of (\ref{eq-gov}) for any $\mathrm{d} \ge 1$.

The main contributions of this paper are summarized below:
\begin{itemize}[leftmargin=*]
  \item We propose a order reduction method that enables rigorous construction and convergence analysis of nonuniform $L1$ schemes for time-delayed fractional sub-diffusion equations.
  \item Under identical regularity assumptions, our nonuniform SYM $L1$ scheme achieves a temporal convergence order of $2-\alpha/2$, strictly superior to the classical $2-\alpha$ rate achieved by direct $L1$ discretization of the original equation.
  \item The SYM $L1$ scheme admits a natural and stable extension to an ADI scheme for two-dimensional delayed fractional sub-diffusion problems, thereby enhancing computational efficiency without compromising accuracy or theoretical guarantees.
  \item We propose a new framework for analyzing the error of the $L1$ scheme  on graded grids, which is particularly well-suited for cases with small $\alpha$. In particular, our error bounds accurately predict that milder (compared to the optimal) grading yields optimal convergence rates at continuous points.
\item The local error analysis presented in \cite{Delay07-Bu,Delay05-Bu,Delay02-Cen} is restricted to at most first-order convergence. While our local error analysis attains the theoretically optimal convergence order of $2-\alpha/2$.
\end{itemize}

The structure of this paper is as follows. The idea of SFOR method and stability of the equivalent models are investigated in Section \ref{sec-SFOR}. In Section \ref{sec-na}, two types of numerical algorithms are constructed to solve the considered models. In Section \ref{sec-Local}, the local error analysis of SYM $L1$ scheme is developed. Numerical experiments are carried out to verify our theoretical results in Section \ref{sec-num}. Finally, some conclusions are drawn in Section \ref{conclusion}.

\section{The SFOR method}\label{sec-SFOR}
\subsection{The proof of Lemma \ref{SFOR-main}}
This section is devoted to the proof idea of Lemma \ref{SFOR-main}, a key ingredient for the present work.
\begin{lemma}\label{SFOR-main}
For $\alpha\in(0,1)$ and $\partial_t^{\frac{\alpha}{2}} u(0)=0$, it holds that
\begin{align}
\partial_t^{\alpha} u(t)=\partial_t^{\frac{\alpha}{2}}\big(\partial_t^{\frac{\alpha}{2}} u(t)\big).
\end{align}
\end{lemma}
\begin{proof}
Taking $v(t)=\partial_t^{\frac{\alpha}{2}} u(t)$, from the composition property \cite[Lemma 1.3]{Kubica}, $\mathcal{I}^{p}\mathcal{I}^{q}g(t)=\mathcal{I}^{p+q}g(t)$, $p$, $q>0$, one has
\begin{align*}
\mathcal{I}^{1-\frac{\alpha}{2}}v(t)
&= \mathcal{I}^{1-\frac{\alpha}{2}}\partial_t^{\frac{\alpha}{2}} u(t)=  \mathcal{I}^{1-\frac{\alpha}{2}}\mathcal{I}^{1-\frac{\alpha}{2}} u'(t)= \mathcal{I}^{2-\alpha} u'(t)\\
&=\int_0^t\omega_{2-\alpha}(t-s)u'(s)ds\\
&=u(s)\omega_{2-\alpha}(t-s)|_0^t+\int_0^t\omega_{1-\alpha}(t-s)u(s)ds\\
&=-u(0)\omega_{2-\alpha}(t)+\int_0^t\omega_{1-\alpha}(t-s)u(s)ds.
\end{align*}
Using $\partial_t^{\frac{\alpha}{2}}v(t)=\frac{d}{dt}\mathcal{I}^{1-\frac{\alpha}{2}}v(t)$ as $v(0):=\partial_t^{\frac{\alpha}{2}} u(0)=0$ \cite[Lemma 1.1]{Kubica}, we have
\begin{align*}
\partial_t^{\frac{\alpha}{2}}\big(\partial_t^{\frac{\alpha}{2}}u(t)\big)&=\partial_t^{\frac{\alpha}{2}}v(t)=\frac{d}{dt}\mathcal{I}^{1-\frac{\alpha}{2}}v(t)=\frac{d}{dt}\bigg(-u(0)\omega_{2-\alpha}(t)+\int_0^t\omega_{1-\alpha}(s)u(t-s)ds\bigg)\\
&=-u(0)\omega_{1-\alpha}(t)+u(0)\omega_{1-\alpha}(t)+\int_0^t\omega_{1-\alpha}(s)u_t(t-s)ds\\
&=\int_0^t\omega_{1-\alpha}(s)u_t(t-s)ds\\
&:=\partial_t^{\alpha}u(t).
\end{align*}
The proof completes.
\end{proof}

Denote $\beta:=\frac{\alpha}{2}$. By Lemma \ref{SFOR-main}, the model (\ref{eq-gov}) can be equivalently solved by the following coupled equations:

\begin{equation}\label{eq-gov-trans}
\begin{cases}
\partial_t^{\beta} v - \Delta u + u(t-\tau) = f(\textbf{x},t), & (\textbf{x},t) \in \Omega \times (0,K\tau], \\
v=\partial_t^{\beta}u,
 & (\textbf{x},t) \in \Omega \times (0,K\tau], \\
u(\textbf{x},t) = a_0(\textbf{x},t), v(\textbf{x},0)=0, & (\textbf{x},t) \in \bar{\Omega}\times[-\tau,0], \\   u(\textbf{x},t) = 0, & (\textbf{x},t) \in \partial \Omega \times (0,K\tau].
\end{cases}
\end{equation}
One can observe that, by utilizing the proposed SFOR method, the explicit orders of the
Caputo derivatives in the first two equalities of the resulting coupled system (\ref{eq-gov-trans}) are all $\frac{\alpha}{2}$. According to the regularity (\ref{regularity}) of $u$, we assume that the solution $v$
satisfies the following regularity:
\begin{align}\label{regularity_v} \Big|\frac{d v(t)}{d t}\Big|\leq C\big(1+(t-k\tau)^{(k+1)\beta-1}\big)
\end{align}
for $t\in(k\tau,(k+1)\tau],~k=0,1,2,\ldots,K-1.$

\subsection{Stability for system (\ref{eq-gov-trans})}
Next, we prove the stability of the coupled system (\ref{eq-gov-trans}). A useful property of fractional derivative is presented. The proof can be found in Theorem 3.2 in \cite{Kubica}. Let $L^2(\Omega)$ be the square-integrable function space with inner product $(\cdot,\cdot)_{L^2(\Omega)}$ (or $(\cdot,\cdot)$ for short).
\begin{lemma}\label{Coercivity}{\rm(Coercivity)} For any function $h\in W^{1,1}(0,K\tau)$ with $\alpha\in(0,1)$, one has the inequality
$$\frac{2}{\Gamma{(\alpha)}}\int_0^t(t-s)^{\alpha-1}h(s)\partial_t^\alpha h(s)ds\geq h^2(t)-h^2(0).$$
\end{lemma}

\begin{lemma}
For the stability of {\rm(\ref{eq-gov-trans})}, one has
$$\|v\| + \|\nabla u\|\leq C(\|\nabla a_0\|+\|u_{-\tau}\|+\|f\|),$$
where $C$ is a constant.
\end{lemma}
\begin{proof}
Take the inner product $(\cdot,\cdot)_{L^2(\Omega)}$ with $v$ and $\Delta u$ for the first two equations of (\ref{eq-gov-trans}), respectively. It gives that
\begin{align*}
(\partial_t^{\beta} v,v) - (\Delta u,v) + (u_{-\tau},v) &=(f,v), \\
(v,\Delta u)  &= (\partial_t^{\beta} u,\Delta u).
\end{align*}
Adding above equations, we have
\begin{align*}
(\partial_t^{\beta} v,v) + (\partial_t^{\beta} \nabla u,\nabla u)
&= (f,v)- (u_{-\tau},v) \\
&\le(\|u_{-\tau}\|+\|f\|)(\|v\|+\|\nabla u\|).
\end{align*}
By Lemma \ref{Coercivity}, one has
\begin{align*}
\|v\|^2 + \|\nabla u\|^2
\le& \|\nabla a_0\|^2+(\|u_{-\tau}\|+\|f\|)(\|v\|+\|\nabla u\|).
\end{align*}
The desired result follows by the triangle inequality.
\end{proof}

\section{Numerical Algorithms}\label{sec-na}
\subsection{Preliminary}
Next, we divide time interval and space interval as follows. For positive constants $M_1,M_2$, let $h_1=\frac{L_1}{M_1},~x_i=ih_1,~0\leq i \leq M_1,~h_2=\frac{L_2}{M_2},~y_j=jh_2,~0\leq j \leq M_2.$ Let $\bar{\Omega}_h=\{\textbf{x}_h=(x_i,y_j)|0\leq i\leq M_1,~0\leq j\leq M_2\}, ~\Omega_h=\bar{\Omega}_h\cap\Omega$ and $\partial\Omega_h=\bar{\Omega}_h\cap\partial\Omega$. For any grid function $u_h=\{u_{ij}|0\leq i \leq M_1,~0\leq j \leq M_2\}$ and it subspace
$u_{0h} = \left\{ u \in u_h \mid u_{0,j} = u_{M_1, j} = 0, \, 0 \leq j \leq M_2; \, u_{i,0} = u_{i,M_2} = 0, \, 0 \leq i \leq M_1 \right\}$, we use standard finite difference operator $\Delta_h=\delta_x^2+\delta_y^2$ on $\bar{\Omega}_h$ to discretize the Laplace operator $\Delta$, where
$$\delta_xu_{i+\frac12,j}=
 \frac{u_{i+1,j}-u_{i,j}}{h_1},~~
 \delta_x^2u_{i,j}=\frac{\delta_xu_{i+1/2,j}-\delta_xu_{i-1/2,j}}{h_1}.$$
The notations $\delta_yu_{i,j+\frac12}$ and $\delta_y^2u_{i,j}$ can be defined similarly, thus we omit here. For grid functions $u,v\in u_h$, the discrete inner products and norms are defined as follows:
\begin{align*}
 &(u,v)=h_1h_2\sum\limits_{i=1}^{M_1-1}\sum\limits_{j=1}^{M_2-1}u_{i,j}v_{i,j},~~
 \|u\|^2=( u,u),~~\|u\|_\infty=\max_{0\leq i\leq M_1,~0\leq j\leq M_2}|u_{i,j}|,\\
 &( \delta_xu,\delta_xv)
 =h_1h_2\sum\limits_{i=1}^{M_1}\sum\limits_{j=1}^{M_2-1}\delta_xu_{i-\frac12,j}\delta_xv_{i-\frac12,j},~~ \|\delta_x u\|^2=( \delta_xu,\delta_xu),\\
&\|\nabla_hu\|=\sqrt{\|\delta_xu\|^2+\|\delta_yu\|^2}.
 %~~|u|_1=\sqrt{\|\delta_xu\|^2+\|\delta_yu\|^2},\\
 %~~\|\Delta_hu\|^2=h_1h_2\sum\limits_{i=1}^{M_1}\sum\limits_{j=1}^{M_2-1}(\Delta_hu_{ij})^2.
\end{align*}
Corresponding notations in $y$ direction are defined similarly.

For a positive integer $N$, the interval $[-\tau,K\tau]$ is divided into $2(K+1)N$ subintervals with $-\tau=t_0<t_1<\cdots< 0=t_{2N}<t_{2N+1}<\cdots<t_{2 (K+1)N}=K\tau$. Denote time step sizes $\rho_n=t_n-t_{n-1},~1\leq n \leq 2(K+1)N.$ Due to the time derivative discontinuity of $u$ near $t=0,\tau,2\tau,\ldots,(K-1)\tau,$ a family of graded meshes will be used, which has dense points near the point of time delay and sparse points in other places. More precisely, we assume that for a fixed parameter $\gamma\geq1$,
\begin{align}\label{b-1}
\begin{array}{c}
 t_{0}=-\tau,~t_{n}=\left\{
 \begin{array}{ll}
 \dfrac{\tau}{2}\Big(\dfrac{n-pN}{N}\Big)^\gamma+\big(\frac{p}{2}-1\big)\tau,~~pN+1\leq n<(p+1)N,\\
 -\dfrac{\tau}{2}\Big(\dfrac{(2+p)N-n}{N}\Big)^\gamma+\frac{p}{2}\tau,~~(p+1)N\leq n\leq (p+2)N,
 \end{array}\right.
 \end{array}
 \end{align}
where $p=0,2,\ldots,2K$.
Particularly, the step sizes of nonuniform meshes (\ref{b-1}) on $[-0.5,1]$ are given by Figure \ref{figure2}.
\begin{figure}[!htbp]
\setlength{\abovecaptionskip}{0.28cm}
\begin{center}
\includegraphics[height=4.5cm,width=6.5cm]{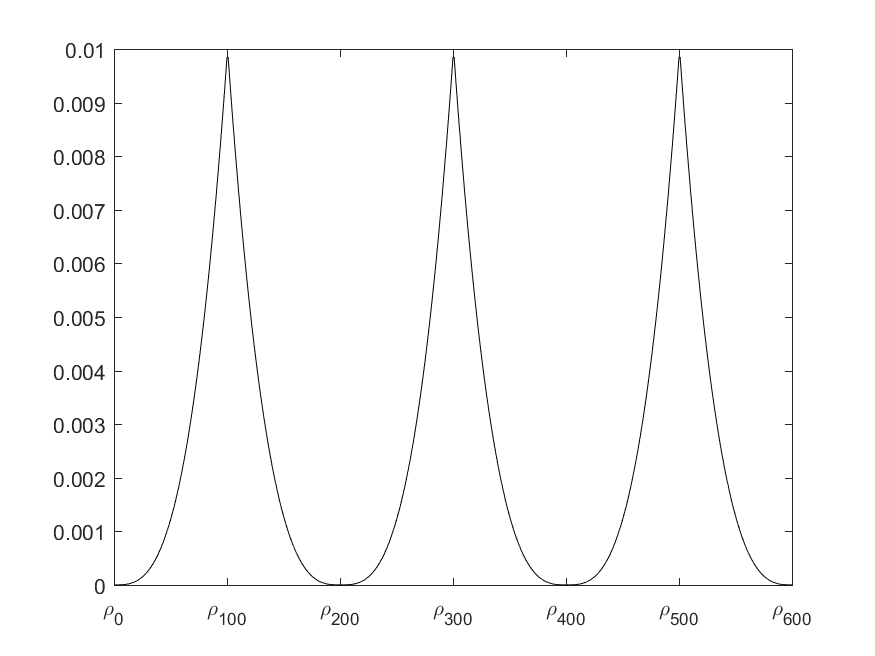}
\end{center}
\caption{The step sizes of nonuniform meshes (\ref{b-1}) with $\gamma=4,~N=100$ \cite{OuCNSNS}.}\label{figure2}
\end{figure}

It is easy to check that (\ref{b-1}) satisfies
\begin{align}\label{t-2}
\rho:=t_{2N+1}\simeq N^{-\gamma},~ t_{2\ell N+j}-t_{2\ell N}\simeq \rho j^\gamma,~ \rho_k\leq C(\rho_{2N+1})^{\frac{1}{\gamma}}(t_k-(\ell-1)\tau)^{1-\frac{1}{\gamma}},
\end{align}
where $2\ell N+1\leq k\leq2(\ell+1)N,~~\ell=1,2,\ldots,K.$ Specially,
from Figure \ref{figure2}, we can observe that $\rho_{2\ell N+1}=\rho,~\ell=1,2,\ldots,K$.
 Denote $\varrho=\max\limits_{2N+1\leq 2(K+1)N-1}\rho_n/\rho_{n+1}$.

The $L1$ formula \cite{Sun1} is used to approximate Caputo derivative, then we have
\begin{align}\nonumber
\bar\partial_t^{{\beta}}u^n
=&\sum\limits_{k=2N+1}^{n}\int_{t_{k-1}}^{t_k}\frac{(t_n-z)^{-{\beta}}}{\Gamma{(1-\beta)}}
\frac{u(t_k)-u(t_{k-1})}{\rho_k}dz\\\label{dd-1}
=&\sum\limits_{k=2N+1}^{n}A_{n-k}^{(n)}[u(t_k)-u(t_{k-1})],~~2N+1\leq n \leq 2(K+1)N,
\end{align}
where $A_{n-k}^{(n)}=\int_{t_{k-1}}^{t_k}\frac{(t_n-z)^{-{\beta}}}{\Gamma{(1-\beta)}\rho_k}dz$ and $A_{0}^{(n)}\ge A_{1}^{(n)}\ge\cdots\ge A_{n-1}^{(n)}>0$. Denote
 \begin{align}\label{b-3}
\begin{array}{c}
P_{n-j}^{(n)}=\left\{
\begin{array}{ll}
\frac{1}{A_{0}^{(n)}},\quad j=n,\\
\frac{1}{A_{0}^{(j)}}\sum\limits_{k=j+1}^{n}
(A_{k-j-1}^{(k)}-A_{k-j}^{(k)})P_{n-k}^{(n)},
\quad 2N+1\leq j\leq n-1,
\end{array}\right.
\end{array}
\end{align}
which are chosen to enforce the identity
\begin{align}\label{b-4}
\sum\limits_{j=k}^{n}P_{n-j}^{(n)}A_{j-k}^{(j)}\equiv1~~ \mbox{for}~2N+1\leq k\leq n.
\end{align}

\iffalse
\textbf{Nonuniform $L1^+$ formula}
We approximate the Caputo derivative at time $t=t_{n-1/2}$, called the $L1^+$ formula, as follows
\begin{align}\label{L1+}
(\bar\partial_t^\beta v)^{n-\frac12}:=\frac{1}{\tau_n} \int_{t_{n-1}}^{t_n}\int_0^t
\omega_{1-\beta}(t-s)(\Pi_1v)' (s) dsdt =\sum\limits_{k=1}^{n}A_{n-k}^{(n)}\nabla_{\tau}v^k~~\mbox{for} ~n\geq 1,
\end{align}
in which the discrete convolution kernels $A^{(n)}_{n-k}$ are given by
\[A^{(n)}_{n-k}:= \frac{1}{\tau_n\tau_k}\int_{t_{n-1}}^{t_n}\int_{t_{k-1}}^{\min\{t,t_k\}}
\omega_{1-\beta}(t-s)dsdt~~\mbox{for} ~1\leq k\le n.\]
\fi

Next, some crucial lemmas are given as follows.
\begin{lemma}\label{P}
The discrete sequence $P_{n-k}^{(n)}$  satisfies
$$0<P_{n-k}^{(n)}<\Gamma(2-\beta)\rho_k^\beta,~2N+1\le k\le n,$$
and
$$\sum_{k=2\ell N+1}^nP_{n-k}^{(n)}\big(t_k-(\ell-1)\tau\big)^{-{\beta}}\leq C~~for~1\leq\ell\leq K.$$
\end{lemma}

\begin{lemma}\label{lem-r}
Let $\alpha\in(0,1).$ Denote $\mathcal{R}_u^n=\partial_t^{\beta}u^n-\bar\partial_t^{\beta}u^n$ and $\mathcal{R}_v^n=\partial_t^{\beta}v^n-\bar\partial_t^{\beta}v^n$. Under the regularity assumptions {\rm(\ref{regularity})} and {\rm(\ref{regularity_v})}, it holds that
$$\|\nabla_h\mathcal{R}_u^n\|\leq C\Big(\rho/\big(t_n-(\ell-1)\tau\big)\Big)^{\min\{\alpha+1, (2-\beta)/\gamma\}}$$
$$|\mathcal{R}_v^n|\leq C\Big(\rho/\big(t_n-(\ell-1)\tau\big)\Big)^{\min\{\beta+1,(2-\beta)/\gamma\}}$$
and
$$\max\limits_{2N+1\leq k\leq n}\Big(\sum_{p=1}^{\ell-1}\sum_{j=2pN+1}^{2(p+1)N}P_{k-j}^{(k)}\|\nabla_h\mathcal{R}_u^j\|
+\sum_{j=2\ell N+1}^{n}P_{k-j}^{(k)}\|\nabla_h\mathcal{R}_u^j\|\Big)
\leq CN^{-\min\{\gamma\alpha,2-\beta\}}$$
$$\max\limits_{2N+1\leq k\leq n}\Big(\sum_{p=1}^{\ell-1}\sum_{j=2pN+1}^{2(p+1)N}P_{k-j}^{(k)}|\mathcal{R}_v^j|
+\sum_{j=2\ell N+1}^{n}P_{k-j}^{(k)}|\mathcal{R}_v^j|\Big)
\leq CN^{-\min\{\gamma\beta,2-\beta\}}$$
for $2\ell N+1\leq n\leq2(\ell+1)N,~1\leq \ell\leq K$.
\end{lemma}
The above lemmas are similar to (\cite{OuCNSNS,Delay01-Tan}).
\begin{lemma}{\rm(\cite{SunZ})}\label{L-3}
 If $f(x)\in C^4[x_{i-1},x_{i+1}],~1\leq i \leq M_1-1$, then it holds that
\begin{align*}
 f''(x_i)=\frac{1}{h_1^2}[f(x_{i-1})-2f(x_i)+f(x_{i+1})]
 -\frac{h_1^2}{12}f^{(4)}(\xi_i),
\end{align*}
where $\xi_i\in(x_{i-1},x_{i+1})$.
\end{lemma}

Some important lemmas are introduced as following.
\begin{lemma}\label{FD-ineq}
{\rm(\cite{LiaoH2018L1})} For $v^n$, $1\leq n\leq N$, one has
\begin{align*}
(\bar\partial_t^{\beta} v^{n},v^{n})\geq\frac{1}{2}\sum\limits_{k=2N+1}^{n}A_{n-k}^{(n)}\nabla_{\tau}
(\|v^k\|^2),\quad 2N+1\leq n\leq 2(K+1)N.
\end{align*}
\end{lemma}

We also need a discrete fractional Gr\"{o}nwall inequality which is a slightly modified version of {\rm (\cite{Delay01-Tan})}.

\begin{lemma}\label{grownwall}
Let $\{g^n\}_{n=2N+1}^{2(K+1)N}$ and $\{\lambda_n\}_{n=0}^{2KN-1}$ be given nonnegative sequences. Assume that $\lambda$ is a positive constant such that $\lambda\geq\sum_{n=0}^{2KN-1}\lambda_n$ and the maximum time-step size satisfies
$$\max\limits_{2N+1\leq n\leq2(K+1)N}\rho_n\leq \frac{1}{\sqrt[\beta]{4\Gamma{(2-\beta)}\lambda}}.$$
Then, for any nonnegative sequence $\{u^k\}_{k=2N}^{2(K+1)N}$ and $\{v^k\}_{k=2N}^{2(K+1)N}$ such that
$$\sum_{k=2N+1}^{n}A_{n-k}^{(n)}\nabla_\tau\big[(u^k)^2+(v^k)^2\big]\leq \sum_{k=2N+1}^{n}\lambda_{n-k}\big(u^{k}+v^{k}\big)^2+\big(u^{n}+v^{n}\big)g^{n},~~2N+1\leq n\leq 2(K+1)N,$$
it holds that
$$u^n+v^n\leq 4E_{\beta,1}\big(4\max\{1,\varrho\}\lambda t_n^{\beta}\big)\Big(u^{2N}+v^{2N}+\max\limits_{2N+1\leq k\leq n}\sum_{j=2N+1}^{k}P_{k-j}^{(k)}g^{j}\Big),~~2N+1\leq n\leq 2(K+1)N,$$
where $E_{\beta,1}(z)=\sum_{k=0}^{\infty}\frac{z^k}{\Gamma{(\beta k+1)}}$.
\end{lemma}

Next, the numerical schemes and corresponding convergence analysis will be given.

\subsection{Convergence of the 1D case }
The corresponding numerical scheme (SYM $L1$ scheme) of problems (\ref{eq-gov-trans}) is as follows:
\begin{equation}\label{num-scheme-1}
\begin{cases}
  \bar\partial_t^{\beta} V_h^{n} - \Delta_h U_h^{n} +U_h^{n-2N}= f_h^n, &\textbf{x}_h\in\Omega_h,~ 2N+1\leq n \leq 2(K+1)N, \\
  V_h^{n}  = \bar\partial_t^{\beta} U_h^{n}, & \textbf{x}_h\in\Omega_h,~2N+1\leq n \leq 2(K+1)N, \\
  U_h^n = a_0(\textbf{x}_h,t_n),~ V_h^0=0, & \textbf{x}_h\in \bar{\Omega}_h, ~1\leq n \leq 2N,  \\
  U_h^n = V_h^n=0, & \textbf{x}_h \in \partial \Omega_h,~~2N+1\leq n \leq 2(K+1)N,
\end{cases}
\end{equation}
where $U_h^n$ and $V_h^n$ are numerical solutions corresponding to $u(\textbf{x}_h,t_n)$ and $v(\textbf{x}_h,t_n)$ in (\ref{eq-gov-trans}).
Next, the convergence analysis
 of scheme (\ref{num-scheme-1}) will be given. Let $u_h^n=u(\textbf{x}_h,t_n)$ and $v_h^n=v(\textbf{x}_h,t_n)$. Denote the solution errors $\bar u_h^n:=u_h^n-U_h^n$ and $\bar v_h^n:=v_h^n-V_h^n$. One has the error system of scheme (\ref{num-scheme-1}):
\begin{equation}\label{error-system-1}
\begin{cases}   \bar\partial_t^{\beta} \bar v_h^{n} - \Delta_h \bar u_h^{n} + \bar u_h^{n-2N}= \mathcal{R}_v^{n}+R_s^n, &\textbf{x}_h\in\Omega_h,~ 2N+1\leq n \leq 2(K+1)N, \\
\bar v_h^{n}  = \bar\partial_t^{\beta} \bar u_h^{n} + \mathcal{R}_u^{n}, & \textbf{x}_h\in\Omega_h,~2N+1\leq n \leq 2(K+1)N, \\
\bar u_h^n = \bar v_h^0=0, & \textbf{x}_h\in \bar{\Omega}_h, ~1\leq n \leq 2N,\\
\bar u_h^n = \bar v_h^n=0, &\textbf{x}_h\in \partial\Omega_h ,~2N+1\leq n \leq 2(K+1)N.
\end{cases}
\end{equation}
The error estimate of $\mathcal{R}_u^{n}$ and $\mathcal{R}_v^{n}$ can be derived by Lemma \ref{lem-r}, and $R_s=\mathcal{O}(h_1^2)$.

\begin{theorem}\label{thm-conv-1}
The numerical scheme {\rm(\ref{error-system-1})} is unconditional convergent with
\[
\|\bar v_h^n\|+\|\nabla_h \bar u_h^n\| \le
C(N^{-\min\{\gamma\beta,2-\beta\}}+N^{-\min\{\gamma\alpha,2-\beta\}}+h_1^2)
\]
for $2N+1\leq n \leq 2(K+1)N$.
\end{theorem}
\begin{proof}
Acting $\nabla_h$ on the second equation of (\ref{error-system-1}).
Taking the inner product with $\bar v_h^n$ and $-\nabla_h \bar u_h^n$ for the first two equations, respectively. It gives that
\begin{align*}
&(\bar\partial_t^{\beta} \bar v_h^{n},\bar v_h^{n}) - (\Delta_h \bar u_h^{n},\bar v_h^{n}) + (\bar u_h^{n-2N},\bar v_h^{n})= (\mathcal{R}_v^{n}+R_s^n,\bar v_h^{n}), \\
&-(\nabla_h\bar v_h^{n},\nabla_h \bar u_h^{n})
= -(\bar\partial_t^{\beta} \nabla_h\bar u_h^{n},\nabla_h \bar u_h^{n})-(\nabla_h\mathcal{R}_u^{n},\nabla_h \bar u_h^{n}).
\end{align*}
Adding above equations, it comes that
\begin{align*}
(\bar\partial_t^{\beta}  \bar v_h^{n},\bar v_h^{n}) + (\bar\partial_t^{\beta}  \nabla_h \bar u_h^{n},\nabla_h \bar u_h^{n}) + (\bar u_h^{n-2N},\bar v_h^{n}) = (\mathcal{R}_v^{n}+R_s^n,\bar v_h^{n})-(\nabla_h \mathcal{R}_u^{n},\nabla_h \bar u_h^{n}).
\end{align*}
By Lemma \ref{FD-ineq}, one has
\begin{align*}
&\frac12\bar\partial_t^{\beta} \big(\|\bar v_h^n\|^2 + \|\nabla_h \bar u_h^n\|^2\big) \\
\le& \big(\|\mathcal{R}_v^{n}\|+\|R_s^{n}\|+\|\bar u_h^{n-2N}\|\big)\|\bar v_h^{n}\|+\|\nabla_h \mathcal{R}_u^{n}\|\|\nabla_h \bar u_h^{n}\|\\
\le&\big(\|\mathcal{R}_v^{n}\|+\|R_s^{n}\|+\|\bar u_h^{n-2N}\|+\|\nabla_h \mathcal{R}_u^{n-\theta}\|\big)\big(\|\bar v_h^{n}\|+\|\nabla_h \bar u_h^{n}\|\big).
\end{align*}
The desired result follows from Lemmas \ref{P} and \ref{grownwall}.
\end{proof}

\iffalse
\begin{lemma}\label{grownwall}
(\cite{LyuP2022SFOR})Let $(g^n)_{n=1}^N$ and $(\lambda_l)_{l=0}^{N-1}$ be given nonnegative sequences. Assume that there exists a constant $\Lambda$ such that $\Lambda\geq\sum_{l=0}^{N-1}\lambda_l$, and that the maximum step satisfies
$$\max_{1\leq n \leq N}\tau_n\leq \frac{1}{\sqrt[\beta]{4\pi_A\Gamma{(2-\beta)}\Lambda}}.$$
Then, for any nonnegative sequence $(v^k)_{k=0}^N$ and $(w^k)_{k=0}^N$ satisfying
$$\sum_{k=1}^nA_{n-k}^{(n)}\nabla_\tau\big[(v^k)^2+(w^k)^2\big]\leq\sum_{k=1}^n\lambda_{n-k}\big(v^{k-\theta}+w^{k-\theta}\big)^2+(v^{n-\theta}+w^{n-\theta})g^n,~~1\leq n\leq N,$$
it holds that
$$v^n+w^n\leq 4E_\beta(4\pi_A\Lambda t_n^\beta)\bigg(v^0+w^0+\max_{1\leq k\leq n}\sum_{j=1}^kP_{k-j}^{(k)}g^j\bigg),~~1\leq n\leq N,$$
where $P_0^{(n)}:=\frac{1}{A_0^{(n)}}$, $P_{n-j}^{(n)}:=\frac{1}{A_0^{(j)}}\sum_{k=j+1}^n(A_{k-j-1}^{(k)}-A_{k-j}^{(k)})P_{n-k}^{(n)}$, $1\leq j\leq n-1$, $E_\beta(z)=\sum_{k=0}^\infty\frac{z^k}{\Gamma{(1+k\beta)}}$ is the Mittag-Leffler function.
\end{lemma}

\begin{lemma}\label{P}
For the sequence $(P_{n-j}^{(n)})_{j=1}^n$, some properties are given in \cite{LiaoH2018L1}.
\begin{align*}
&\sum_{j=k}^nP_{n-j}^{(n)}A_{j-k}^{(j)}\equiv1,~~1\leq k\leq n,\\
&0\leq P_{n-j}^{(n)}\leq \pi_A\Gamma{(2-\beta)}\tau_j^\beta,~~\sum_{j=1}^nP_{n-j}^{(n)}\omega_{1-\beta}(t_j)\leq C,~~ 1\leq j\leq n\leq N.
\end{align*}
\end{lemma}
\fi

\subsection{Convergence of the 2D case}
Next, we write the $L1$ approximation for problem (\ref{eq-gov-trans}) in $\mathbb{R}^2$ as follows:
\begin{equation}\label{num-scheme-2}
\begin{cases}   \bar\partial_t^{\beta} v_h^{n} - \Delta_h u_h^{n} +u_h^{n-2N}= f_h^n+\mathcal{R}_v^{n}+R_s^n, &\textbf{x}_h\in\Omega_h,~ 2N+1\leq n \leq 2(K+1)N, \\   v_h^{n}  = \bar\partial_t^{\beta} u_h^{n}+\mathcal{R}_u^{n}, & \textbf{x}_h\in\Omega_h,~2N+1\leq n \leq 2(K+1)N, \\   u_h^n = a_0(\textbf{x}_h,t_n),~ v_h^0=0, & \textbf{x}_h\in \bar{\Omega}_h, ~1\leq n \leq 2N,  \\   u_h^n = v_h^n=0, & \textbf{x}_h \in \partial \Omega_h,~~2N+1\leq n \leq 2(K+1)N,
\end{cases}
\end{equation}
In order to get an ADI scheme to improve the computational efficiency, we introduce a weighted auxiliary term
$$\varepsilon_h^n=\mu^3\delta_x^2\delta_y^2u_h^n~~\mbox{with}~\mu=\frac{1}{A_0^{(n)}}.$$
It is easy to check that $A_0^{(n)}=\mathcal{O}(\rho_n^{-\beta})$, thus we have
$$\varepsilon_h^n=\mathcal{O}(\rho_n^{3\beta}).$$

Now we rewrite the Eqs.(\ref{num-scheme-2}) as follows:
\begin{equation}\label{num-scheme-3}
\begin{cases}   \bar\partial_t^{\beta} v_h^{n} - \Delta_h u_h^{n} +u_h^{n-2N}= f_h^n+\mathcal{R}_1^{n}, &\textbf{x}_h\in\Omega_h,~ 2N+1\leq n \leq 2(K+1)N, \\   v_h^{n}  = \bar\partial_t^{\beta} u_h^{n}+\varepsilon_h^n+\mathcal{R}_2^{n}, & \textbf{x}_h\in\Omega_h,~2N+1\leq n \leq 2(K+1)N, \\   u_h^n = a_0(\textbf{x}_h,t_n),~ v_h^0=0, & \textbf{x}_h\in \bar{\Omega}_h, ~1\leq n \leq 2N,  \\   u_h^n = v_h^n=0, & \textbf{x}_h \in \partial \Omega_h,~~2N+1\leq n \leq 2(K+1)N,
\end{cases} \end{equation}
where $\mathcal{R}_1^{n}=\mathcal{R}_v^{n}+R_s^n$ and $\mathcal{R}_2^{n}=\mathcal{R}_u^{n}-\varepsilon_h^n$.
We employ $U_h^n$ and $V_h^n$ to represent the numerical solutions of $u(\textbf{x}_h,t_n)$ and $v(\textbf{x}_h,t_n)$, respectively, and take $$\epsilon_h^n=\mu^3\delta_x^2\delta_y^2U_h^n.$$
Omitting the small terms $\mathcal{R}_1^{n}$ and $\mathcal{R}_2^{n}$ in (\ref{num-scheme-3}), we obtain the SYM $L1$-ADI scheme
\begin{equation}\label{num-scheme-4}
\begin{cases}   \bar\partial_t^{\beta} V_h^{n} - \Delta_h U_h^{n} +U_h^{n-2N}= f_h^n, &\textbf{x}_h\in\Omega_h,~ 2N+1\leq n \leq 2(K+1)N, \\   V_h^{n}  = \bar\partial_t^{\beta} U_h^{n}+\epsilon_h^n, & \textbf{x}_h\in\Omega_h,~2N+1\leq n \leq 2(K+1)N, \\   U_h^n = a_0(\textbf{x}_h,t_n),~ V_h^0=0, & \textbf{x}_h\in \bar{\Omega}_h, ~1\leq n \leq 2N,  \\   U_h^n = V_h^n=0, & \textbf{x}_h \in \partial \Omega_h,~~2N+1\leq n \leq 2(K+1)N.
\end{cases}
\end{equation}

For the numerical scheme (\ref{num-scheme-4}), utilizing the second equality to eliminate \( V_{h}^n \) in the first equality, it can be represented by the following ADI sense:
\begin{align*}
\left( A_0^{(n)} -\mu\delta_x^2 \right) \left( A_0^{(n)} - \mu\delta_y^2 \right) U_{h}^n
&= (V_{\text{his}})_{h}^{n-1} -U_h^{n-2N}+ f_{h}^{n}
+ A_0^{(n)} (U_{\text{his}})_{h}^{n-1},\\
V_{h}^n &=  \partial_t^{\beta} U_{h}^{n} + \epsilon_{h}^{n},
\end{align*}
provided initial and boundary conditions in (\ref{num-scheme-4}), where

\[
(U_{\text{his}})_{h}^{n-1} := \sum_{k=2N+1}^{n-1} (A_{n-k-1}^{(n)} - A_{n-k}^{(n)}) U_{h}^k + A_{n-2N-1}^{(n)}U_{h}^{2N},
\]
\[
(V_{\text{his}})_{h}^{n-1} := \sum_{k=2N+1}^{n-1} (A_{n-k-1}^{(n)} - A_{n-k}^{(n)})V_{h}^k + A_{n-2N-1}^{(n)}V_{h}^{2N}.
\]

 Denote $\bar u_h^n:=u_h^n-U_h^n$ and $\bar v_h^n:=v_h^n-V_h^n$. One has the error system of scheme (\ref{num-scheme-4}):
\begin{equation}\label{error-system-2}
\begin{cases}   \bar\partial_t^{\beta} \bar v_h^{n} - \Delta \bar u_h^{n}+  \bar u_h^{n-2N}= \mathcal{R}_1^{n}, & \textbf{x}_h\in\Omega_h,~ 2N+1\leq n \leq 2(K+1)N,\\
\bar v_h^{n}  = \bar\partial_t^{\beta} \bar u_h^{n} +\mu^3\delta_x^2\delta_y^2\bar u_h^n + \mathcal{R}_2^{n}, & \textbf{x}_h\in\bar\Omega_h,~ 2N+1\leq n \leq 2(K+1)N, \\
\bar u_h^n = \bar v_h^0=0, & \textbf{x}_h\in\Omega_h,~ 1\leq n \leq 2N, \\
\bar u_h^n = \bar v_h^n=0, & \textbf{x}_h\in\partial\Omega_h,~ 2N+1\leq n \leq 2(K+1)N.
\end{cases}
\end{equation}

\begin{theorem}\label{thm-conv-ADI}
The numerical scheme {\rm(\ref{error-system-2})} is unconditional convergent with \[ \|\bar v_h^n\|+\|\nabla_h \bar u_h^n\| \le C(N^{-\min\{\gamma\beta,\frac{3}{2}\alpha\}}+h_1^2+h_2^2) \] for $2N+1\leq n \leq 2(K+1)N$. \end{theorem}
\begin{proof}
Acting $\nabla_h$ on the second equation of (\ref{error-system-2}). Taking the inner product with $\bar v_h^n$ and $-\nabla_h \bar u_h^n$ for the first two equations, respectively. It gives that
\begin{align*}
&(\bar\partial_t^{\beta} \bar v_h^{n},\bar v_h^{n}) - (\Delta_h \bar u_h^{n},\bar v_h^{n}) + (\bar u_h^{n-2N},\bar v_h^{n})= (\mathcal{R}_1^{n},\bar v_h^{n}), \\
&-(\nabla_h\bar v_h^{n},\nabla_h \bar u_h^{n})  = -(\bar\partial_t^{\beta} \nabla_h\bar u_h^{n},\nabla_h \bar u_h^{n})-(\mu^3\nabla_h \delta_x^2\delta_y^2\bar u_h^n,\nabla_h \bar u_h^{n})
-(\nabla_h\mathcal{R}_1^{n},\nabla_h \bar u_h^{n}).
\end{align*}
Adding above equations, it comes that
\begin{align*} &(\bar\partial_t^{\beta}  \bar v_h^{n},\bar v_h^{n}) + (\bar\partial_t^{\beta}  \nabla_h \bar u_h^{n},\nabla_h \bar u_h^{n}) + (\bar u_h^{n-2N},\bar v_h^{n})\\
= &(\mathcal{R}_1^{n},\bar v_h^{n})-\mu^3(\nabla_h \delta_x^2\delta_y^2\bar u_h^n,\nabla_h \bar u_h^{n})-(\nabla_h \mathcal{R}_2^{n},\nabla_h \bar u_h^{n}).
\end{align*}
Note that
\[
\mu^3 (\nabla_h \delta_x^2 \delta_y^2 \bar u_h^n, \nabla_h \bar u_h^n)
= \mu^3 (\nabla_h \delta_x \delta_y\bar u_h^n, \nabla_h \delta_x \delta_y \bar u_h^n) \geq 0.
\]
By Lemma \ref{FD-ineq}, one has \begin{align*} &\frac12\bar\partial_t^{\beta} \big(\|\bar v_h^n\|^2 + \|\nabla_h \bar u_h^n\|^2\big) \\ \le& \big(\|\mathcal{R}_1^{n}\|+\|\bar u_h^{n-2N}\|\big)\|\bar v_h^{n}\|+\|\nabla_h \mathcal{R}_2^{n}\|\|\nabla_h \bar u_h^{n}\|\\ \le&\big(\|\mathcal{R}_1^{n}\|+\|\bar u_h^{n-2N}\|+\|\nabla_h \mathcal{R}_2^{n}\|\big)\big(\|\bar v_h^{n}\|+\|\nabla_h \bar u_h^{n}\|\big).
\end{align*}
The desired result follows from Lemmas \ref{P} and \ref{grownwall}.
\end{proof}

\subsection{Spatial compact scheme for 2D case}

We notice that the ADI method could be combined with the spatial high-order compact approximation, thereby developing a numerical scheme with spatial fourth-order accuracy.
Define the compact operator
\[
\mathcal{H}_x u_{ij} := \frac{1}{12} (u_{i-1,j} + 10u_{ij} + u_{i+1,j}), \quad 1 \leq i \leq M_1 - 1, \quad 0 \leq j \leq M_2,
\]
and \(\mathcal{H}_y\) is defined similarly. We have \cite{HiaoHL-ADI}
\[
\mathcal{H}_x [u_{xx}(x_i, y_j)] = \delta_x^2 u_{ij} + \mathcal{O}(h_1^4), \quad \mathcal{H}_y [u_{yy}(x_i, y_j)] = \delta_y^2 u_{ij} + \mathcal{O}(h_2^4).
\]

%Define $\mathcal{H} u_{ij}=\mathcal{H}_x\mathcal{H}_y u_{ij},\text{ and } \|u^n\|_{H}=(-\mathcal{H}u,u)^{\frac12}$ for $u\in u_h$.
Taking the operator \(\mathcal{H}_x \mathcal{H}_y\) on both sides of the first equality in (\ref{num-scheme-3}), replacing the auxiliary term \(\varepsilon_{h}^{n}\) in the second equality of (\ref{num-scheme-3}) by \(\hat{\varepsilon}_{h}^{n} := \mu^3 \mathcal{H}_y^{-1} \mathcal{H}_x^{-1} \delta_x^2 \delta_y^2 u_{h}^{n}\), and omitting the small errors, it further obtains the following compact weighted ADI scheme:

\begin{align}\nonumber
&\left( A_0^{(n)} \mathcal{H}_x - \mu \delta_x^2 \right) \left( A_0^{(n)} \mathcal{H}_y -\mu \delta_y^2 \right) U_{h}^{n} \\\label{COM-ADI1}
= & \mathcal{H}_x \mathcal{H}_y \left( (V_{\text{his}})_{h}^{n-1} + f_{h}^{n}-U_h^{n-2N} \right)
+ A_0^{(n)} \mathcal{H}_x \mathcal{H}_y  (U_{\text{his}})_{h}^{n-1}, \\\label{COM-ADI2}
&V_{h}^n = \bar{\partial}_{t}^{\beta} U_{h}^{n} + \mu^3 \mathcal{H}_y^{-1} \mathcal{H}_x^{-1} \delta_x^2 \delta_y^2 U_{h}^{n},
\end{align}
for \( 2N+1 \leq n \leq 2(K+1)N \), equipped with initial and boundary conditions in (\ref{num-scheme-4}).

\begin{remark}
 Owing to the positive definiteness of the operators \(\mathcal{H}_x, \mathcal{H}_y, \mathcal{H}_x^{-1}\) and \(\mathcal{H}_y^{-1}\), following the derivations of Theorem \ref{thm-conv-ADI}, it is easy to see that the  SYM $L1$-ADI compact scheme {\rm(\ref{COM-ADI1})-(\ref{COM-ADI2})} is unconditionally convergent with the accuracy of \(\mathcal{O}(N^{-\min\{\gamma\beta,\frac{3}{2}\alpha\}}+h_1^4 +h_2^4)\) under analogous solution regularity.
\end{remark}

\section{The local error analysis of the proposed schemes}\label{sec-Local}
\subsection{Stability properties of the L1 discrete fractional derivative operator}
The global error analysis imposes strict limitations on the time discretization, particularly requiring a sufficiently small step size to ensure convergence. This restriction becomes  especially restrictive when the fractional order $\beta$ is small. This motivates our rigorous, localized error analysis for the proposed methods applied to such problems. It is noting  that we write \( a \simeq b \) when \( a \lesssim b \) and \( a \gtrsim b \), and \( a \lesssim b \) when \( a \leq Cb \) with a generic constant \( C \) depending on \( \Omega, T, u^0, v^0\) and \( f \), but not on the total numbers of grids in space or time.
Firstly, we define the complementary discrete kernels \( G_{n,j} \) given in \cite{ChenHu2022} by
\[
G_{n,n} = 1, \quad G_{n,j} = \sum_{k=1}^{n-j} \rho_{n-k}^{\beta} (A_{k-1}^{(n)} - A_{k}^{(n)}) G_{n-k,j},\]
for $2N+1\leq n\leq2(K+1)N \text{ and } 2N+1 \leq j \leq n-1.$
And $A_{k-1}^{(n)}\ge A_{k}^{(n)}$ implies $G_{n,j}>0$ for all $n,j$.

For the positive multipliers \( G_{n,j} \), one has the following inequality.
\begin{lemma}\label{lem-G}
For \( n = 2N+1, \ldots, 2(K+1)N \), we have
\[
\Gamma(2 - \beta) \rho_n^\beta \sum_{j=2N+1}^n G_{n,j} \leq \frac{t_n^\beta}{\Gamma(1 + \beta)}.
\]
\end{lemma}

\begin{lemma}%{\rm (\cite[Lemma 3.4]{ChenHu})}
\label{lem-B}
Let \(\{\mathrm{g}^j\}_{j=2N}^{2(K+1)N}\) be arbitrary mesh functions. Set \(B^{2N} = 0\) and \(B^n = \Gamma(2 - \beta)\rho_n^\beta \sum_{j=2N+1}^n G_{n,j}\mathrm{g}^j\) for \(n = 2N+1, \ldots, 2(K+1)N\). Then one has
\[
\bar\partial_t^\beta B^n = \Gamma(2-\beta)\mathrm{g}^n \text{ for } n = 2N+1, 2N+2, \ldots, 2(K+1)N.
\]
\end{lemma}

\begin{proof}
Using the definitions of \(B^n\) and \(G_{n,j}\), for each \(n\) we get
\[
\begin{aligned}
\bar\partial_t^\beta B^n
&= A^{(n)}_0B^n - \sum_{k=1}^{n-2N-1} B^{n-k}(A^{(n)}_{k-1} - A^{(n)}_{k}) \\
&= A^{(n)}_0B^n - \Gamma(2-\beta)\sum_{k=1}^{n-2N-1} (A^{(n)}_{k-1} - A^{(n)}_{k})\rho_{n-k}^{\beta} \sum_{j=2N+1}^{n-k} G_{n-k,j}\mathrm{g}^j \\
&= A^{(n)}_0B^n - \Gamma(2-\beta)\sum_{j=2N+1}^{n-1} \sum_{k=1}^{n-j}\rho_{n-k}^{\beta} (A^{(n)}_{k-1} - A^{(n)}_{k})G_{n-k,j}\mathrm{g}^j \\
&= A^{(n)}_0B^n - \Gamma(2-\beta)\sum_{j=2N+1}^{n-1} G_{n,j}\mathrm{g}^j = \Gamma(2-\beta)G_{n,n}\mathrm{g}^n = \Gamma(2-\beta)\mathrm{g}^n,
\end{aligned}
\]
since \(A^{(n)}_0 = \rho_n^{-\beta}\).
\end{proof}

\begin{theorem}\label{theorem-M}
Let \( 1 \leq \gamma \leq (2 - \beta)/\beta \).
Assume the nonnegative grid function sequences \(\{v^n\}_{n=2N}^{2(K+1)N}\) satisfies \( v^{2N} = 0 \) and
\[
\bar\partial_t^\beta v^n \lesssim \sum_{l=1}^{\ell}\big(t_n-(l-1)\tau\big)^{(l-1)\beta}\big(\rho/(t_n-(l-1)\tau)\big)^{\eta_l+1}\] for $ 2\ell N+1\leq n\leq2(\ell+1)N,~1\leq \ell\leq K.$
Then
\[
v^n \leq C \mathcal{M}^n_{\eta_\ell} \quad \text{for } 2\ell N+1\leq n\leq2(\ell+1)N,~1\leq \ell\leq K,
\]
where
\begin{align}\label{M}
\mathcal{M}^n_{\eta_\ell} :=
\begin{cases}
\sum\limits_{l=1}^{\ell}\rho(t_n-(l-1)\tau\big)^{\beta-1}& \text{if } \eta_l > 0, \\
%\le \sum\limits_{l=1}^{\ell}\rho^{1-\beta}(t_n-(l-1)\tau\big)^{\beta-1}
\sum\limits_{l=1}^{\ell}\rho(t_n-(l-1)\tau\big)^{\beta-1}\big[1 + \ln\big((t_n-(\ell-1)\tau)/\rho\big)\big] & \text{if } \eta_l = 0, \\ \sum\limits_{l=1}^{\ell}\rho^{\eta_l+1}(t_n-(\ell-1)\tau\big)^{\beta-\eta_l-1}  \text{ if } \eta_1=\cdots=\eta_\ell \text{ and} &\beta-1\le\eta_l<0,
  \\
\sum\limits_{l=1}^{\ell}\rho^{\eta_l+1}(t_n-(l-1)\tau\big)^{\beta-\eta_l-1} & \text{if } \eta_l <  \beta-1.
\end{cases}
\end{align}
\end{theorem}
\begin{proof}
The proof is presented in the following Lemmas \ref{case-1}-\ref{case-4}, where a few cases will be considered separately.
\end{proof}

\begin{lemma}[\textbf{Proof of Theorem \ref{theorem-M} for $\eta_l\le \beta-1$}]\label{case-1}
If $\eta_l\le \beta-1$, then (\ref{M}) holds true on an arbitrary temporal mesh $\{t_j\}_{j=2N}^n$.
\end{lemma}
\begin{proof}
According to \cite[Lemma 2.1(ii)]{Kopteva2019}, for any nonnegative \(\{v^j\}_{j=2N}^{2(K+1)N}\) on an arbitrary mesh \(\{t_j\}_{j=2N}^{2(K+1)N}\), one has
\begin{align}\label{stability}
v^n \lesssim \max_{2N+1\le j\le n} \left\{t_j^\beta \sum_{l=1}^{\ell}\big(t_j-(l-1)\tau\big)^{(l-1)\beta}\big(\rho/(t_j-(l-1)\tau)\big)^{\eta_l+1}\right\}
\end{align}
for $ 2\ell N+1\leq j\le n\leq2(\ell+1)N,~1\leq \ell\leq K.$
The assumptions on $\{v^n\}$ yield
\begin{align*}
&t_j^\beta \sum_{l=1}^{\ell}\big(t_j-(l-1)\tau\big)^{(l-1)\beta}\big(\rho/(t_j-(l-1)\tau)\big)^{\eta_l+1}\\
=&t_j^\beta \sum_{l=1}^{\ell-1}\big(t_j-(l-1)\tau\big)^{(l-1)\beta}\big(\rho/(t_j-(l-1)\tau)\big)^{\eta_l+1}
+t_j^\beta \big(t_j-(\ell-1)\tau\big)^{(\ell-1)\beta}\big(\rho/(t_j-(\ell-1)\tau)\big)^{\eta_\ell+1}\\
\lesssim&\sum_{l=1}^{\ell-1}\rho^{\eta_l+1}\big(t_j-(l-1)\tau\big)^{l\beta-\eta_l-1}
+t_j^\beta \big(t_j-(\ell-1)\tau\big)^{(\ell-1)\beta-\eta_\ell-1}\rho^{\eta_\ell+1}\\
\lesssim&\sum_{l=1}^{\ell}\rho^{\eta_l+1}\big(t_j-(l-1)\tau\big)^{\beta-\eta_l-1}\lesssim\sum_{l=1}^{\ell}\rho^{\eta_l+1}\big(t_n-(l-1)\tau\big)^{\beta-\eta_l-1}\quad \forall j\le n.
\end{align*}
The desired assertion $v^n\lesssim\sum_{l=1}^{\ell}\rho^{\eta_l+1}\big(t_n-(l-1)\tau\big)^{\beta-\eta_l-1}$
for $ 2\ell N+1\leq n\leq2(\ell+1)N,~1\leq \ell\leq K$
follows.
\end{proof}
\subsubsection{Proof of Theorem \ref{theorem-M} I}
In this case($\eta_l\geq \beta$) $(\rho/(t_j-(l-1)\tau))^{\eta_l+1}\le (\rho/(t_j-(l-1)\tau))^{\beta+1}$, so it suffices to consider only $\eta_l= \beta$. For the latter case, as the operator $\bar\partial_t^\beta$ is associated with a M-matrix, it suffices to prove the following lemma.
That is, if $\bar\partial_t^\beta|V^j|\le\bar\partial_t^\beta W^j\text{ and } W^{2N}=0$, then we have $|V^j|\le|W^j|.$
\begin{lemma}\label{case-2} Let the temporal mesh satisfy \eqref{t-2} with \(1\le \gamma\le (2-\beta)/\beta\).
Then there exists a discrete barrier function \(\{\mathbf{B}^j\}_{j=2N}^{2(K+1)N}\) such that
\[
\mathbf{B}^{2 N}=0,\quad 0\le \mathbf{B}^j\lesssim \sum_{l=1}^{\ell}(t_j-(l-1)\tau)^{\beta-1},\quad
\bar\partial_t^\beta \mathbf{B}^j\gtrsim  \sum_{l=1}^{\ell}\rho^\beta(t_j-(l-1)\tau)^{-\beta-1}
\]
for $ 2\ell N+1\leq j\leq2(\ell+1)N,~1\leq \ell\leq K$.
\end{lemma}
\begin{proof}
Set \(\nu=1-\beta\). Choose a sufficiently large number  \(2\le p_\ell-2\ell N\lesssim 1\).
Define the continuous barrier function
\begin{align}\label{definition_B}
\mathbf{B}(s)=\sum_{l=1}^{\ell}\min\left\{\frac{s-(l-1)\tau}{t_{p_\ell}-(l-1)\tau}\,(t_{p_\ell}-(l-1)\tau)^{-\nu},\;(s-(l-1)\tau)^{-\nu}\right\}
\end{align}
and set \(\mathbf{B}^j=\mathbf{B}(t_j)\) for all grid nodes. Note that, when using the notation of type $\lesssim$, the dependence on $p_\ell$ will be shown
explicitly.
First of all, consider the case $\ell=1$ (i.e.
 $2N+1\le j\le 4N$). From  \cite[(1.2)]{Kopteva2020}, we have
\[\bar\partial_t^\beta \mathbf{B}^j\gtrsim  \rho^\beta t_j^{-\beta-1}.  \]

Next, we consider the case $j\geq 4N+1$.
For \(j\le p_\ell\), we have \(\mathbf{B}(s)=\sum_{l=1}^\ell(s-(l-1)\tau)(t_{p_\ell}-(l-1)\tau)^{-\nu-1}\) is linear on \([0,t_{p_\ell}]\).
Since the $L1$ scheme is exact for linear functions,
\begin{align*}
\bar\partial_t^\beta \mathbf{B}^j = &\partial_t^\beta \mathbf{B}(t_j)
\simeq\sum_{l=1}^{\ell}t_j^{\nu}(t_{p_\ell}-(l-1)\tau)^{-\nu-1}\\
\simeq&\sum_{l=1}^{\ell-1}t_j^{\nu}(t_{p_\ell}-(l-1)\tau)^{-\nu-1}+(p_{\ell}-2\ell N)^{-\gamma(2-\beta)}t_j^{\nu}\rho^{-\nu-1},
\end{align*}
%\sum_{l=1}^{\ell-1}\Big(\frac{t_{p_\ell}-(\ell-1)\tau}{t_{p_\ell}-(l-1)\tau}\Big)^{\nu+1}t_j^{\nu}(t_{p_\ell}-(\ell-1)\tau)^{-\nu-1}
where we also used $t_{p_\ell}-(\ell-1)\tau\simeq\rho(p_\ell-2\ell N)^\gamma$.
As $t_j-(l-1)\tau\ge \rho$, then we get
\begin{align*}
\bar\partial_t^\beta \mathbf{B}^j \gtrsim& \sum_{l=1}^{\ell-1}(t_{j}-(l-1)\tau)^{-1}+(p_{\ell}-2\ell N)^{-\gamma(2-\beta)}\rho^{-\nu-1}(t_j-(\ell-1)\tau)^{\nu}\\
\gtrsim& \sum_{l=1}^{\ell-1}\rho^{\eta_l}(t_j-(l-1)\tau)^{-\eta_l-1}+(p_{\ell}-2\ell N)^{-\gamma(2-\beta)}\rho^{-\nu}(t_j-(\ell-1)\tau)^{\nu-1}\\
\ge& \sum_{l=1}^{\ell-1}\rho^{\eta_l}(t_j-(l-1)\tau)^{-\eta_l-1}+(p_{\ell}-2\ell N)^{-\gamma(2-\beta)}\rho^{\eta_\ell}(t_j-(\ell-1)\tau)^{-\eta_\ell-1}%When t_j-(\ell-1)\tau=\rho, \rho^{-\nu}(t_j-(\ell-1)\tau)^{\nu-1}=\rho^{-1}, which can be incorporated into the final result and [(t_j-(\ell-1)\tau)/\rho]^{\eta_\ell+\nu}\ge 1, since \eta_\ell+\nu>0.
\end{align*}
$\forall \eta_l\ge 0 \text{ including } \eta_l=\beta$ (Noting that the second term in the right hand side also holds for $\forall \eta_l\ge \beta-1$).

For \(t>t_{p_\ell}\), we have \(\mathbf{B}(t)=\sum_{l=1}^\ell(t-(l-1)\tau)^{-\nu}\) on \([t_{p_\ell},t]\), but note that the history \([0,t_{p_\ell}]\) contributes a linear part.
A direct computation gives
\begin{align*}
&\Gamma(1-\beta)\partial_t^\beta \mathbf{B}(t)\\
=&\sum_{l=1}^{\ell}\Big[\underbrace{\int_0^{t_{p_\ell}}(t-s)^{-\beta}(t_{p_\ell}-(l-1)\tau)^{-\nu-1}ds}_{=:I_1}  -\underbrace{\nu\int_{t_{p_\ell}}^t (t-s)^{-\beta}(s-(l-1)\tau)^{-\nu-1}ds}_{=:t^{-1}I_2}\Big].
\end{align*}
Using the substitution \(\hat{s}=s/t,~\hat{t}_{p_\ell}=t_{p_\ell}/t,~\hat{\tau}=\tau/t\) and noting that $\nu+\beta=1$, one obtains
\begin{align*}
I_1=&\frac{1}{\nu}(t_{p_\ell}-(l-1)\tau)^{-\nu-1}(t^{\nu}-(t-t_{p_\ell})^{\nu})\\
=&\frac{1-(1-\hat{t}_{p_\ell})^{\nu}}{\nu}\frac{t^{\nu}(t-(l-1)\tau)^{\beta+1}}{(t_{p_\ell}-(l-1)\tau)^2}(t_{p_\ell}-(l-1)\tau)^{\beta}(t-(l-1)\tau)^{-\beta-1}\\
\end{align*}
and
\begin{align*}
t^{-1}I_2=&\nu t^{-1}\int_{\hat{t}_{p_\ell}}^{1}(1-\hat{s})^{-\beta}(\hat{s}-(l-1)\hat{\tau})^{-\nu-1}d\hat{s}\\
=&\frac{t^{-1}}{1-(l-1)\hat{\tau}}(1-\hat{t}_{p_\ell})^\nu(\hat{t}_{p_\ell}-(l-1)\hat{\tau})^{-\nu}\\
\le&\frac{t^{-1}}{1-(l-1)\hat{\tau}}(\hat{t}_{p_\ell}-(l-1)\hat{\tau})^{-\nu}(1-\nu\hat{t}_{p_\ell})\\
=&\frac{t^{\nu}}{t-(l-1)\tau}(t_{p_\ell}-(l-1)\tau)^{-\nu}-\frac{\nu t_{p_\ell}}{t-(l-1)\tau}(t_{p_\ell}-(l-1)\tau)^{-\nu}t^{-\beta}\\
=&\Big[\frac{t^{\nu}(t-(l-1)\tau)^{\beta}}{t_{p_\ell}-(l-1)\tau}-\frac{\nu t_{p_\ell}}{t_{p_\ell}-(l-1)\tau}\Big(\frac{t-(l-1)\tau}{t}\Big)^{\beta}\Big](t_{p_\ell}-(l-1)\tau)^{\beta}(t-(l-1)\tau)^{-\beta-1}.
\end{align*}
Therefore, it yields that
\begin{align}\nonumber
&\Gamma(1-\beta)\partial_t^\beta \mathbf{B}(t)\\\nonumber \ge&\sum_{l=1}^{\ell}\Big[\frac{1-(1-\hat{t}_{p_\ell})^{\nu}}{\nu}\frac{t^{\nu}(t-(l-1)\tau)^{\beta+1}}{(t_{p_\ell}-(l-1)\tau)^2}-\frac{t^{\nu}(t-(l-1)\tau)^{\beta}}{t_{p_\ell}-(l-1)\tau}\\\nonumber
&+\frac{\nu t_{p_\ell}}{t_{p_\ell}-(l-1)\tau}\Big(\frac{t-(l-1)\tau}{t}\Big)^{\beta}\Big](t_{p_\ell}-(l-1)\tau)^{\beta}(t-(l-1)\tau)^{-\beta-1}\\\nonumber
=&\sum_{l=1}^{\ell}\Big[\Big(\frac{(1-(1-\hat{t}_{p_\ell})^{\nu})(t-(l-1)\tau)}{\nu(t_{p_\ell}-(l-1)\tau)}-1\Big)\frac{t^{\nu}(t-(l-1)\tau)^{\beta}}{t_{p_\ell}-(l-1)\tau}\\\label{Caputo}
&+\frac{\nu t_{p_\ell}}{t_{p_\ell}-(l-1)\tau}\Big(\frac{t-(l-1)\tau}{t}\Big)^{\beta}\Big](t_{p_\ell}-(l-1)\tau)^{\beta}(t-(l-1)\tau)^{-\beta-1}.
\end{align}

For the sake of completing the proof, it remains to show that $|\bar\partial_t^\beta \mathbf{B}^n-\partial_t^\beta \mathbf{B}(t_n)|\le \frac12 \partial_t^\beta \mathbf{B}(t_n)$ for any $n>p_\ell$. Now we estimate the difference between the discrete and continuous derivatives.
Note that $\Gamma(1-\beta)\bigl[\bar\partial_t^\beta \mathbf{B}^n-\partial_t^\beta \mathbf{B}(t_n)\bigr]=\sum_{j=2N+1}^{n}\mu^j$, where use the piecewise linear interpolant \(\mathbf{B}^I\) of \(\mathbf{B}\),
\begin{align}
\mu^j=\int_{t_{j-1}}^{t_{j}}(t_n-s)^{-\beta}\bigl(\mathbf{B}(s)-\mathbf{B}^I(s)\bigr)ds
=\beta\int_{t_{j-1}}^{t_{j}}(t_n-s)^{-\beta-1}\bigl(\mathbf{B}(s)-\mathbf{B}^I(s)\bigr)ds.
\end{align}
The interpolation error is zero on intervals where \(\mathbf{B}\) is linear (i.e., for \(j\le p_\ell\)).
For \(p_\ell+1\le j\le n-1\), standard error estimates give
\[
|\mathbf{B}(s)-\mathbf{B}^I(s)|\lesssim \rho_j^2|\mathbf{B}''(t_{j-1})|.
\]
For $j=n$, we shall use a similar but sharper bound $|\mathbf{B}(s)-\mathbf{B}^I(s)|\lesssim \rho_j(t_n-s)|\mathbf{B}''(t_{j-1})|$.
Combining these yields $|\mathbf{B}(s)-\mathbf{B}^I(s)|\lesssim \rho_j^2\min\{1,(t_n-s)/\rho_n\}|\mathbf{B}''(t_{j-1})|$ for $j>p_\ell$, where $|\mathbf{B}''(t_{j-1})|\lesssim|\mathbf{B}''(s)|\simeq\sum_{l=1}^{\ell}(s-(l-1)\tau)^{-\nu-2}$. %(in view of $t_{j-1}\simeq t_{j}$).
Using the mesh condition (\ref{t-2}): \(\rho_j\simeq\rho^{1/\gamma}(t_j-(\ell-1)\tau)^{1-1/\gamma}\simeq\rho^{1/\gamma}(s-(\ell-1)\tau)^{1-1/\gamma}\le \rho^{1/\gamma}(s-(l-1)\tau)^{1-1/\gamma}\), we arrive at
\[
|\mu^j|\lesssim \rho^{2/\gamma}\sum_{l=1}^{\ell}\int_{t_{j-1}}^{t_j}(t_n-s)^{-\beta-1}(s-(l-1)\tau)^{-\nu-2/\gamma}\min\left\{1,(t_n-s)/\rho_n\right\}ds,\quad  j>p_\ell.
\]
This immediately yields the bound
\begin{align*}
&|\bar\partial_t^\beta \mathbf{B}^n-\partial_t^\beta \mathbf{B}(t_n)|\\
\lesssim&\rho^{2/\gamma}\sum_{l=1}^{\ell}\int_{t_{p_\ell}}^{t_n}(t_n-s)^{-\beta-1}(s-(l-1)\tau)^{-\nu-2/\gamma}\min\left\{1,(t_n-s)/\rho_n\right\}ds\\
\lesssim&\rho^{2/\gamma}t_n^{-2/\gamma-1}\sum_{l=1}^{\ell}\int_{\hat{t}_{p_\ell}}^{1}(1-\hat{s})^{-\beta-1}(\hat{s}-(l-1)\hat{\tau})^{-\nu-2/\gamma}\min\left\{1,(1-\hat{s})/\hat{\rho}_n\right\}d\hat{s},
\end{align*}
where \(\hat{s}=s/t_n\), \(\hat{t}_{p_\ell}=t_{p_\ell}/t_n\), \(\hat{\rho}_n=\rho_n/t_n\) and \(\hat{\tau}_n=\tau_n/t_n\).
Here, when bounding the integral, it is convenient to split into three parts: \((\hat{t}_{p_\ell},\max\{\tfrac12,\hat{t}_{p_\ell}\})\), \((\max\{\tfrac12,\hat{t}_{p_\ell}\},1-\hat{\rho}_n)\) and \((1-\hat{\rho}_n,1)\), where $\hat{\rho}_n\le \frac12$ if $p_\ell$ is large enough.
Estimating each part separately gives
\begin{align*}
&\int_{\hat{t}_{p_\ell}}^{1}(1-\hat{s})^{-\beta-1}(\hat{s}-(l-1)\hat{\tau})^{-\nu-2/\gamma}\min\left\{1,(1-\hat{s})/\hat{\rho}_n\right\}d\hat{s}\\
\le&\int_{\hat{t}_{p_\ell}}^{\max\{\tfrac12,\hat{t}_{p_\ell}\}}(\hat{s}-(l-1)\hat{\tau})^{-\nu-2/\gamma}d\hat{s}+\int_{\frac12}^{1-\hat{\rho}_n}(1-\hat{s})^{-\beta-1}d\hat{s}+\hat{\rho}_n^{-1}\int_{{1-\hat{\rho}_n}}^{1}(1-\hat{s})^{-\beta}d\hat{s}\\
\lesssim& (\hat{t}_{p_\ell}-(l-1)\hat{\tau})^{\beta-2/\gamma}+\hat{\rho}_n^{-\beta}.
\end{align*}
Notice that \(\hat{\rho}_n=\rho_n/t_n\simeq t_n^{-1}\rho^{1/\gamma}\big(t_n-(\ell-1)\tau\big)^{1-1/\gamma}\). Finally, in view of $1\le \gamma\le(2-\beta)/\beta$, one has $\big(\rho/(t_n-(l-1)\tau)\big)^{(2-\beta)/\gamma}\lesssim \big(\rho/(t_n-(l-1)\tau)\big)^{\beta}$. Therefore
\begin{align*}
&|\bar\partial_t^\beta \mathbf{B}^n-\partial_t^\beta \mathbf{B}(t_n)|\\
 \lesssim & \rho^{2/\gamma}t_n^{-2/\gamma-1}\sum_{l=1}^{\ell}\Bigl[(\hat{t}_{p_\ell}-(l-1)\hat{\tau})^{\beta-2/\gamma}
 +\Big(\frac{t_n-(\ell-1)\tau}{t_n}\Big)^{-\beta}\big(\rho/(t_n-(l-1)\tau)\big)^{-\beta/\gamma}\Bigr]\\
=&\sum_{l=1}^{\ell}\Big[\rho^{2/\gamma}t_n^{-\beta-1}(t_{p_\ell}-(l-1)\tau)^{\beta-2/\gamma}
+\rho^{2/\gamma}t_n^{-2/\gamma-1}\Big(\frac{t_n-(\ell-1)\tau}{t_n}\Big)^{-\beta}\big(\rho/(t_n-(l-1)\tau)\big)^{-\beta/\gamma}\Big]\\
=&\sum_{l=1}^{\ell}\Big(\frac{\rho}{t_{p_\ell}-(l-1)\tau}\Big)^{2/\gamma}\Big(\frac{t_n-(l-1)\tau}{t_n}\Big)^{\beta+1}(t_{p_\ell}-(l-1)\tau)^\beta (t_n-(l-1)\tau)^{-\beta-1}\\ &+\sum_{l=1}^{\ell}t_n^{-1}\Big(\frac{t_n-(l-1)\tau}{t_n}\Big)^{2/\gamma}\Big(\frac{t_n-(\ell-1)\tau}{t_n}\Big)^{-\beta} \big(\rho/(t_n-(l-1)\tau)\big)^{(2-\beta)/\gamma}\\
\lesssim&\sum_{l=1}^{\ell}\Big(\frac{\rho}{t_{p_\ell}-(l-1)\tau}\Big)^{2/\gamma}\Big(\frac{t_n-(l-1)\tau}{t_n}\Big)^{\beta+1}(t_{p_\ell}-(l-1)\tau)^\beta (t_n-(l-1)\tau)^{-\beta-1}\\ &+\sum_{l=1}^{\ell}\Big(\frac{t_{p_\ell}-(l-1)\tau}{t_n}\Big)^{2/\gamma+1-\beta}\Big(\frac{t_{n}-(l-1)\tau}{t_{p_\ell}-(l-1)\tau}\Big)^{2/\gamma+1}  \Big(\frac{\rho}{t_{n}-(\ell-1)\tau}\Big)^{\beta}(t_{p_\ell}-(l-1)\tau)^\beta (t_n-(l-1)\tau)^{-\beta-1}\\
\le&\sum_{l=1}^{\ell}\Big[\Big(\frac{\rho}{t_{p_\ell}-(l-1)\tau}\Big)^{2/\gamma}+\Big(\frac{t_{n}-(l-1)\tau}{t_{p_\ell}-(l-1)\tau}\Big)^{2/\gamma+1} \Big(\frac{\rho}{t_{n}-(\ell-1)\tau}\Big)^{\beta}\Big](t_{p_\ell}-(l-1)\tau)^\beta (t_n-(l-1)\tau)^{-\beta-1}.
\end{align*}
Combining the above result with (\ref{Caputo}) and choosing $p_\ell$ sufficiently large yields
\begin{align}\nonumber
\bar\partial_t^\beta \mathbf{B}^n \ge& \partial_t^\beta \mathbf{B}(t_n)-|\bar\partial_t^\beta \mathbf{B}^n-\partial_t^\beta \mathbf{B}(t_n)|\\\nonumber
\gtrsim&\sum_{l=1}^{\ell}(t_{p_\ell}-(l-1)\tau)^\beta (t_n-(l-1)\tau)^{-\beta-1}\\\label{bound}
\ge&\sum_{l=1}^{\ell}(t_{p_\ell}-(\ell-1)\tau)^\beta (t_n-(l-1)\tau)^{-\beta-1}
\quad\forall n>p_\ell,
\end{align}
and hence $\forall n\geq 4N+1.$
This completes the proof.
\end{proof}

\subsubsection{Proof of Theorem \ref{theorem-M} for $\beta-1\le\eta_l< \beta$}
We shall use the notations and some findings from the proof of Lemma \ref{case-2}. Particularly, $\nu = 1-\beta$, while $p_\ell-2\ell N \simeq 1$ was chosen sufficiently large in the proof of Lemma \ref{case-2}. When using the notation of type $\lesssim$, the dependence on $\eta_l$ and $m$, but not on $p_\ell$, will be shown explicitly.
For $m \ge 0$ and $\eta_l <\beta$, set
\begin{align}\label{B-1}
&(p_\ell)_m -2\ell N:= 2^{m} (p_\ell-2\ell N),\\
&(\mathbf{B}_l)^j_m := \min\left\{ \frac{t_j-(l-1)\tau}{(t_{(p_\ell)_m}-(l-1)\tau)^{\nu+1}}, (t_j-(l-1)\tau)^{-\nu} \right\},\quad \mathbf{B}^j_m := \sum_{l=1}^{\ell} (\mathbf{B}_l)^j_m,\\\label{B-2}
&c_m := 2^{-m\gamma\eta_\ell} \quad\Rightarrow\quad c_m (t_{(p_\ell)_m}-(\ell-1)\tau )^{\eta_\ell} \simeq \rho^{\eta_\ell},
\end{align}
for $2\ell N+1\le j\le 2(\ell+1)N,~\ell=1,2,\ldots,K.$
Here the final observation follows from (\ref{t-2}) (which yields $t_{(p_\ell)_m}-(\ell-1)\tau \simeq \rho ((p_\ell)_m-2\ell N)^\gamma$).

Note that $\mathbf{B}^j_{0} = \mathbf{B}^j$, and more generally, $\mathbf{B}^j_m = \mathbf{B}^j\big|_{p_\ell:=(p_\ell)_m}$, where $\mathbf{B}^j$ is from (\ref{definition_B}). Conveniently, in the proof of Lemma \ref{case-2}, the dependence on any sufficiently large $p_\ell$ was shown explicitly. In particular, we recall that $\bar\partial_t^\beta \mathbf{B}^j_m \ge 0$ for $j \ge 2N+1$. Furthermore,
\begin{align}\label{p_0}
\bar\partial_t^\beta \mathbf{B}^j_{0} \gtrsim \sum_{l=1}^{\ell}\rho^{\eta_l} (t_j-(l-1)\tau)^{-\eta_l-1} \quad \text{for } 2N+1 \le j \le (p_\ell)_{0}.
\end{align}
The relation for $\mathbf{B}^j_0 = \mathbf{B}^j$ can be found in the above-mentioned proof for $\eta_l \ge 0$ (in fact, $\bar\partial_t^\beta (\mathbf{B}_\ell)^j_{0} \gtrsim \rho^{\eta_\ell} (t_j-(\ell-1)\tau)^{-\eta_\ell-1}$ holds for $\eta_\ell\ge \beta-1$). From (\ref{bound}), one has $\bar\partial_t^\beta \mathbf{B}^j_m \gtrsim \sum_{l=1}^\ell\big(t_{(p_\ell)_m}-(l-1)\tau\big)^\beta\big(t_{j}-(l-1)\tau\big)^{-\beta-1}$. \iffalse Since the solution in the historic interval $(0,(\ell-1)\tau](\ell\ge 2)$ is smooth, then we have \begin{align} \nonumber
\sum_{l=1}^{\ell-1}\bar\partial_t^\beta (\mathbf{B}_l)^j_m
\gtrsim& \sum_{l=1}^{\ell-1}\big(t_{(p_\ell)_m}-(l-1)\tau\big)^\beta\big(t_{j}-(l-1)\tau\big)^{-\beta-1}\\\nonumber
\ge& \sum_{l=1}^{\ell-1}\big(t_{(p_\ell)_m}-(l-1)\tau\big)^{\eta_l}\big(t_{j}-(l-1)\tau\big)^{-\eta_l-1}\\\label{p_history}
\simeq&\sum_{l=1}^{\ell-1}\rho^{\eta_l} \big(t_{j}-(l-1)\tau\big)^{-\eta_l-1} \text{ for } (p_\ell)_m < j \le(p_\ell)_{m+1}.
\end{align}\fi
Indeed, it implies
\begin{align} \nonumber
&\sum_{l=1}^{\ell-1}\bar\partial_t^\beta (\mathbf{B}_l)^j_m+c_m\bar\partial_t^\beta (\mathbf{B}_\ell)^j_m\\\nonumber
\gtrsim &\sum_{l=1}^{\ell-1}\big(t_{(p_\ell)_m}-(l-1)\tau\big)^\beta\big(t_{j}-(l-1)\tau\big)^{-\beta-1}+ c_m \big(t_{(p_\ell)_m}-(\ell-1)\big)^{\eta_\ell} \big(t_{j}-(\ell-1)\tau\big)^{-\eta_\ell-1} \\\nonumber
\gtrsim& \sum_{l=1}^{\ell-1}\big(t_{(p_\ell)_m}-(l-1)\tau\big)^{\eta_l}\big(t_{j}-(l-1)\tau\big)^{-\eta_l-1}+\rho^{\eta_\ell} \big(t_{j}-(\ell-1)\tau\big)^{-\eta_\ell-1}
\\\label{p_m}
\gtrsim &\sum_{l=1}^{\ell}\rho^{\eta_l} \big(t_{j}-(l-1)\tau\big)^{-\eta_l-1}~ \text{ for } (p_\ell)_m < j \le(p_\ell)_{m+1} \text{ and }\eta_l\ge 0.
\end{align}

Note that $c_m\bar\partial_t^\beta(\mathbf{B}_\ell)^j_m\gtrsim\rho^{\eta_\ell} \big(t_{j}-(\ell-1)\tau\big)^{-\eta_\ell-1}~ \text{ for } (p_\ell)_m < j \le(p_\ell)_{m+1}\text{ and }\eta_l\ge \beta-1$ holds as well.
Now we are prepared to prove the following two lemmas, which are sufficient for establishing Theorem \ref{theorem-M} for $\eta_l \in (0,\beta)$ and $\eta_l \in [\beta-1,0]$ ($l=1,2,\ldots,\ell$) respectively.

\begin{lemma}\label{case-3}
Let the temporal mesh satisfy \eqref{t-2} with \(1\le \gamma\le (2-\beta)/\beta\). Suppose that $\eta_l \in (0,\beta)$. Then there exists a discrete barrier function $\{ \mathbf{\bar B}^j\}_{j=2N}^{2(K+1)N}$ such that $\mathbf{\bar B}^{2N} = 0$, while $0 \le \mathbf{\bar B}^j \lesssim \sum_{l=1}^{\ell}(t_j-(l-1)\tau)^{\beta-1}$ and $\bar\partial_t^\beta \mathbf{\bar B}^j \gtrsim \sum_{l=1}^{\ell}\rho^{\eta_l} (t_j-(l-1)\tau)^{-\eta_l-1}$ for $j \ge 2N+1$.
\end{lemma}

\begin{proof}
Using (\ref{B-1})-(\ref{B-2}), let $\mathbf{\bar B}^j := \sum_{m=0}^\infty c_m \mathbf{B}^j_m$. Then $\bar\partial_t^\beta \mathbf{\bar B}^j \gtrsim \sum_{l=1}^\ell\rho^{\eta_l} \big(t_{j}-(l-1)\tau\big)^{-\eta_l-1}$ for all $j\ge 2N+1$ follows from (\ref{p_0})-(\ref{p_m}), while $\sum_{m=0}^\infty c_m = C_{\eta_\ell}:= (1-2^{-\gamma\eta_\ell})^{-1}$, so $\mathbf{\bar B}^j \le  C_{\eta_\ell} \sum_{l=1}^\ell\big(t_{j}-(l-1)\tau\big)^{\beta-1}$, which completes the proof.
\end{proof}

\begin{lemma}\label{case-4}
Let the temporal mesh satisfy \eqref{t-2} with \(1\le \gamma\le (2-\beta)/\beta\), and suppose that $\eta_l \in [\beta-1, 0]$. If $V^0 = 0$ and
\[
\bar\partial_t^\beta |V^j| \lesssim
\begin{cases}
\sum_{l=1}^{\ell}\rho^{\eta_l} (t_j-(l-1)\tau)^{-\eta_l-1},\quad &\eta_l=0,\\
\sum_{l=1}^{\ell}\rho^{\eta_l} (t_j-(\ell-1)\tau)^{-\eta_l-1},\quad &\eta_l<0 \text{ and } \eta_1=\cdots=\eta_\ell,
\end{cases}
\]
for $2\ell N+1\le j\le n\le 2(\ell+1)N,~\ell=1,2,\ldots,K$, then
\[
|V^n| \lesssim
\begin{cases}
\sum_{l=1}^{\ell}(t_n-(l-1)\tau)^{\beta-1}\big[1+\ln\big((t_n-(\ell-1)\tau)/\rho\big)\big], & \eta_l = 0,\\
\sum_{l=1}^{\ell}(t_n-(\ell-1)\tau)^{\beta-1}(\rho/(t_n-(\ell-1)\tau))^{\eta_l}, & \eta_l < 0 \text{ and } \eta_1=\cdots=\eta_\ell.
\end{cases}
\]
\end{lemma}
\begin{proof}
Using (\ref{B-1})-(\ref{B-2}), let $\mathbf{\bar B}^j := \sum_{m=0}^M \sum_{l=1}^{\ell-1} (\mathbf{\bar B}_l)^j_m+\sum_{m=0}^M c_m (\mathbf{\bar B}_\ell)^j_m$ and $\eta_l = 0$, where $M = 0$ if $n \le p_\ell$, and $M := \lceil \log_2\big((n-2\ell N)/(p_\ell-2\ell N)\big) - 1 \rceil$ otherwise, so that $(p_\ell)_M < n \le (p_\ell)_{M+1}$. Note also that $M \lesssim \ln (n-2\ell N) \simeq \ln\big((t_n-(\ell-1)\tau)/\rho\big)$ (as $(t_n-(\ell-1)\tau)/\rho\simeq (n-2\ell N)^\gamma$ in view of (\ref{t-2})). Then $   \bar\partial_t^\beta\mathbf{\bar B}^j\gtrsim\sum_{l=1}^{\ell}\rho^{\eta_l} \big(t_{j}-(l-1)\tau\big)^{-\eta_l-1}$ for all $j\ge 2\ell N+1$.
Hence $|V^j| \lesssim \mathbf{\bar B}^j$ for all $j\le n$, in particular $|V^n| \lesssim \mathbf{\bar B}^n$.
On the other hand, $\mathbf{\bar B}^j \le \sum_{l=1}^{\ell}(t_j-(l-1)\tau)^{\beta-1} \sum_{m=0}^M c_m$. Since $\eta_\ell = 0$, each $c_m = 1$, so $\sum_{m=0}^M c_m = 1+M \simeq 1+\ln\big((t_n-(\ell-1)\tau)/\rho\big)$, hence $\mathbf{\bar B}^j \lesssim \sum_{l=1}^{\ell}(t_j-(l-1)\tau)^{\beta-1}\big[1+\ln\big((t_n-(\ell-1)\tau)/\rho\big)\big]$.

When $\eta_l \in (\beta-1,0)$, let $\mathbf{\bar B}^j := \sum_{l=1}^{\ell}\sum_{m=0}^M c_m (\mathbf{\bar B}_\ell)^j_m$ and $\eta_1=\cdots=\eta_\ell$, where $M = 0$ if $n \le p_\ell$, and $M := \lceil \log_2\big((n-2\ell N)/(p_\ell-2\ell N)\big) - 1 \rceil$ otherwise, so that $(p_\ell)_M < n \le (p_\ell)_{M+1}$. Note also that $M \lesssim  \ln\big((t_n-(\ell-1)\tau)/\rho\big)$. Then $ \bar\partial_t^\beta\mathbf{\bar B}^j\gtrsim\sum_{l=1}^{\ell}\rho^{\eta_\ell} \big(t_{j}-(\ell-1)\tau\big)^{-\eta_\ell-1}$ for all $j\ge 2\ell N+1$. Therefore, $|V^j| \lesssim \mathbf{\bar B}^j$ for all $j\le n$, in particular $|V^n| \lesssim \mathbf{\bar B}^n$.
Meanwhile, we have $\sum_{m=0}^M c_m = (c_{M+1}-1)/(c_1-1)$, where $c_{M+1} \simeq \big(\rho/(t_{(p_\ell)_{M+1}}-(\ell-1)\tau)\big)^{\eta_\ell} \simeq \big(\rho/(t_n-(\ell-1)\tau)\big)^{\eta_\ell}$,
while $C_{\eta_\ell}:=(c_1-1)^{-1} =(2^{\gamma|\eta_\ell|}-1)^{-1}$, so finally
$|V^n|\lesssim \mathbf{\bar B}^j\le C_{\eta_\ell}\sum_{l=1}^\ell\big(t_{j}-(\ell-1)\tau\big)^{\beta-1} \big(\rho/(t_n-(\ell-1)\tau)\big)^{\eta_\ell}$.
\end{proof}

Next, a property of the complementary discrete kernels \( G_{n,j} \) will be described as below.

\begin{lemma}\label{lem-BM}
For \( n = 2N+1, 2N+2, \dots, 2(K+1)N \), we set \( B^{2N} = 0 \) and
\begin{align*}
B^n =& \Gamma(2-\beta) \rho_n^{\beta}
\Big(\sum_{p=1}^{\ell-1}\sum_{j=2pN+1}^{2(p+1)N}G_{n,j}\sum_{l=1}^{p}\big(t_j-(l-1)\tau\big)^{(l-1)\beta}\big(\rho/(t_j-(l-1)\tau)\big)^{\eta_l+1}\\ &+\sum_{j=2\ell N+1}^{n}G_{n,j}\sum_{l=1}^{\ell}\big(t_j-(l-1)\tau\big)^{(l-1)\beta}\big(\rho/(t_j-(l-1)\tau)\big)^{\eta_l+1}\Big),
\end{align*}
then it holds that

\[
B^n \leq C \mathcal{M}^n_{\eta_\ell} \text{ for } 2\ell N+1\leq n\leq2(\ell+1)N,~1\leq \ell\leq K.
\]

\end{lemma}
\begin{proof}
According to Lemma \ref{lem-B}, one has
\[ \bar\partial_t^\beta B^n = \Gamma(2-\beta)\sum_{l=1}^{\ell}\big(t_n-(l-1)\tau\big)^{(l-1)\beta}\big(\rho/(t_n-(l-1)\tau)\big)^{\eta_l+1}\]
for  $2\ell N+1\leq n\leq2(\ell+1)N,~1\leq \ell\leq K$.
Then the results follow from Theorem \ref{theorem-M}.
\end{proof}

%\subsection{The local error analysis of the proposed schemes}
\subsection{The local error analysis of SYM L1 scheme (\ref{num-scheme-1})}
Next, we will present the optimal local error analysis for our proposed difference scheme (\ref{num-scheme-1}). It is worth noting that the notations $\mathcal{R}_u^n=\partial_t^{\beta}u^n-\bar\partial_t^{\beta}u^n$ and $\mathcal{R}_v^n=\partial_t^{\beta}v^n-\bar\partial_t^{\beta}v^n$ are still used in this subsection, but new estimates will be provided, which are different from those given in Lemma \ref{lem-r}.
In order to get the error analysis of the difference scheme, the following discrete Gr\"{o}nwall inequality is essential.

\begin{lemma}\rm{(\cite[Lemma 4.2]{Stynes2021})}\label{Gronwall-G}
Suppose that the sequences \( \{ \xi^n \}_{n=2N+1}^{2(K+1)N} \), \( \{ \Psi^n \}_{n=2N+1}^{2(K+1)N}\) and the nonnegative sequences \(  \{u^n\}_{n=2N}^{2(K+1)N} \) and \(  \{v^n\}_{n=2N}^{2(K+1)N} \) , and they satisfy
\[
(\bar\partial_t^{\beta} (u^n+v^n)) (u^n+v^n) \leq \xi^n (u^n+v^n) + (\Psi^n)^2,~~2N+1\leq n\leq 2(K+1)N,
\]
then it holds that
\[
u^n+v^n \leq u^0 + v^0+  \Gamma(2-\beta) \rho_n^{\beta} \sum_{j=2N+1}^n G_{n,j} (\xi^j + \Psi^j) + \max_{2N+1 \leq j \leq n} \{ \Psi^j \}.
\]
\end{lemma}

Next we will present the following two properties of the operator \( \bar\partial_t^{\beta} \), which can be found in \cite{ChenHu2022}.

\begin{lemma}\label{lem-norm}
Assume \( u^n \in u_{0h} \) for \( 2N+1 \leq n \leq 2(K+1)N \), then we get
\[
\left( \bar\partial_t^{\beta} u^n, u^n \right) \geq \left(\bar\partial_t^{\beta} \|u^n\|\right) \|u^n\| \quad \text{for } 2N+1 \leq n \leq 2(K+1)N.
\]
\iffalse
and
\[
\left( \bar\partial_t^{\beta} u^n, -\mathcal{H}u^n \right) \geq (\bar\partial_t^{\beta} \|u^n\|_H) \|u^n\|_H \quad \text{for } 2N+1 \leq n \leq 2(K+1)N.
\]\fi
\end{lemma}

\begin{lemma}\label{lem-ine}
Assume that the functions \( u_i^n \in u_{0h} \) are nonnegative for \( i = 1, 2, \ldots, l \), then it holds that
\[
\sum_{i=1}^l (\bar\partial_t^\beta u_i^n) u_i^n \geq \left(\bar\partial_t^\beta \sqrt{\sum_{i=1}^l (u_i^n)^2}\right) \sqrt{\sum_{i=1}^l (u_i^n)^2},
\]
for \( 2N+1 \leq n \leq 2(K+1)N \).
\end{lemma}

\begin{theorem}\label{thm-error}
If \(u^n \in C^4(\Omega)\) and the regularity results {\rm(\ref{regularity})} and {\rm(\ref{regularity_v})} hold, then we get
\begin{align}\label{th-error}
\|\bar v_h^n\|+\|\nabla_h \bar u_h^n\| \leq C t_n^{\beta} h_1^2 + C \mathbb{M}^n \quad \text{for } 2N+1 \leq n \leq 2(K+1)N,
\end{align}
where \(\mathbb{M}^n\) is defined by
\begin{align*}%\label{th-error}
\mathbb{M}^n :=
\begin{cases}
N^{-\gamma}  (t_n-(\ell-1)\tau)^{\beta-1} & \text{if } 1 \leq \gamma < 2 - \beta, \\[4pt]
N^{\beta-2}  (t_n-(\ell-1)\tau)^{\beta-1} [1 + \ln((t_n-(\ell-1)\tau) / \rho)] & \text{if } \gamma = 2 - \beta, \\[4pt]
N^{\beta-2} (t_n-(\ell-1)\tau)^{\beta-(2-\beta)/\gamma}\ln(2N) & \text{if } 2 - \beta<\gamma \le (2-\beta)/\beta,\\[4pt] N^{\beta-2}(t_n-(\ell-1)\tau)^{\beta-(2-\beta)/\gamma}& \text{if } \gamma > (2-\beta)/\beta ,
\end{cases}
\end{align*}
for $2\ell N+1\leq n\leq2(\ell+1)N,~1\leq \ell\leq K$.
\end{theorem}
\begin{proof}
By taking the inner product \((\cdot, \cdot)\) on both sides of the first two equalities in (\ref{error-system-1}) with \(\bar{v}^n\) and \(\Delta_hu^n\), respectively, then we obtain
\begin{align*}
&(\bar\partial_t^{\beta} \bar v_h^{n},\bar v_h^{n}) - (\Delta_h \bar u_h^{n},\bar v_h^{n}) + (\bar u_h^{n-2N},\bar v_h^{n})= (\mathcal{R}_v^{n}+R_s^n,\bar v_h^{n}), \\
&-(\nabla_h\bar v_h^{n},\nabla_h \bar u_h^{n})  = -(\bar\partial_t^{\beta} \nabla_h\bar u_h^{n},\nabla_h \bar u_h^{n})-(\nabla_h\mathcal{R}_u^{n},\nabla_h \bar u_h^{n}).
\end{align*}
Adding above equations, it comes that
\begin{align*} (\bar\partial_t^{\beta}  \bar v_h^{n},\bar v_h^{n}) + (\bar\partial_t^{\beta}  \nabla_h \bar u_h^{n},\nabla_h \bar u_h^{n}) + (\bar u_h^{n-2N},\bar v_h^{n}) = (\mathcal{R}_v^{n}+R_s^n,\bar v_h^{n})-(\nabla_h \mathcal{R}_u^{n},\nabla_h \bar u_h^{n}).
\end{align*}
Now applying Lemma \ref{lem-norm} and Cauchy-Schwartz inequality leads to
\begin{align*} \big(\bar\partial_t^{\beta} \|\bar v_h^n\|\big)\|\bar v_h^n\| + \big(\bar\partial_t^{\beta}\|\nabla_h \bar u_h^n\|\big)\|\nabla_h \bar u_h^n\|
\le \big(\|\mathcal{R}_v^{n}\|+\|R_s^{n}\|+\|\bar u_h^{n-2N}\|\big)\|\bar v_h^{n}\|+\|\nabla_h \mathcal{R}_u^{n}\|\|\nabla_h \bar u_h^{n}\|.
\end{align*}

Next, using Lemma \ref{lem-ine} and H\"{o}lder inequality, it holds that
\[
\begin{aligned}
&\left(\bar\partial_t^{\beta} \sqrt{\|\bar v_h^n\|^2 + \|\nabla_h\bar u_h^n\|^2} \right) \sqrt{\|\bar v_h^n\|^2 + \|\nabla_h\bar u_h^n\|^2} \\
\leq& \left(\|\mathcal{R}_v^{n}\|+\|R_s^{n}\|+\|\bar u_h^{n-2N}\|\right)\|\bar v_h^{n}\|+\|\nabla_h \mathcal{R}_u^{n}\|\|\nabla_h \bar u_h^{n}\| \\
\leq& \sqrt{\left(\|\mathcal{R}_v^{n}\|+\|R_s^{n}\|+\|\bar u_h^{n-2N}\|\right)^2+\|\nabla_h \mathcal{R}_u^{n}\|^2}\sqrt{\|\bar v_h^n\|^2 + \|\nabla_h\bar u_h^n\|^2}.
\end{aligned}
\]
This is equivalent to
\[
\bar\partial_t^{\beta} \sqrt{\|\bar v_h^n\|^2 + \|\nabla_h\bar u_h^n\|^2}  \leq \sqrt{\left(\|\mathcal{R}_v^{n}\|+\|R_s^{n}\|+\|\bar u_h^{n-2N}\|\right)^2+\|\nabla_h \mathcal{R}_u^{n}\|^2}.
\]
Invoking Lemma \ref{Gronwall-G} gives
\begin{align}\nonumber
 \sqrt{\|\bar v_h^n\|^2 + \|\nabla_h\bar u_h^n\|^2}
\leq & \Gamma(2-\beta) \rho_n^{\beta} \sum_{j=2N+1}^n G_{n,j} \sqrt{\left(\|\mathcal{R}_v^{j}\|+\|R_s^{j}\|+\|\bar u_h^{j-2N}\|\right)^2+\|\nabla_h \mathcal{R}_u^{j}\|^2}\\\label{con}
\leq & \Gamma(2-\beta) \rho_n^{\beta} \sum_{j=2N+1}^n G_{n,j} \left(\|\mathcal{R}_v^{j}\|+\|R_s^{j}\|+\|\bar u_h^{j-2N}\|+\|\nabla_h \mathcal{R}_u^{j}\|\right).
\end{align}
Similar to \cite[Lemma 5.4]{QingyangSi}, one has
\begin{align*}
\|\mathcal{R}_v^{j}\|+\|\nabla_h \mathcal{R}_u^{j}\|
\le & C\sum_{l=1}^{\ell}\big(t_j-(l-1)\tau\big)^{(l-1)\alpha}\Big(\rho/\big(t_j-(l-1)\tau\big)\Big)^{\min\{l\alpha+1, (2-\beta)/\gamma\}}\\
&+ C\sum_{l=1}^{\ell}\big(t_j-(l-1)\tau\big)^{(l-1)\beta}\Big(\rho/\big(t_j-(l-1)\tau\big)\Big)^{\min\{l\beta+1, (2-\beta)/\gamma\}}\\
\le & C\sum_{l=1}^{\ell}\big(t_j-(l-1)\tau\big)^{(l-1)\beta}\Big(\rho/\big(t_j-(l-1)\tau\big)\Big)^{\eta_l+1}
\end{align*}
for $2\ell N+1\leq j\le n\leq2(\ell+1)N,~1\leq \ell\leq K$, where $\eta_l=\min\{l\beta+1, (2-\beta)/\gamma\}-1$. Furthermore, applying Lemma \ref{L-3} yields $\|R_s^{j}\|\le Ch_1^2$. Thus, (\ref{con}) simplifies to
\begin{align*}
\sqrt{\|\bar v_h^n\|^2 + \|\nabla_h\bar u_h^n\|^2}
\leq & C\Gamma(2-\beta) \rho_n^{\beta}\bigg( \sum_{j=2N+1}^n G_{n,j} \big(h_1^2+\|\bar u_h^{j-2N}\|\big)\\
&+ \sum_{p=1}^{\ell-1}\sum_{j=2pN+1}^{2(p+1)N}G_{n,j}\sum_{l=1}^{p}\big(t_j-(l-1)\tau\big)^{(l-1)\beta}\big(\rho/(t_j-(l-1)\tau)\big)^{\eta_l+1}\\
&+\sum_{j=2\ell N+1}^{n}G_{n,j}\sum_{l=1}^{\ell}\big(t_j-(l-1)\tau\big)^{(l-1)\beta}\big(\rho/(t_j-(l-1)\tau)\big)^{\eta_l+1}\bigg)\\
\leq&  C t_n^{\beta} h_1^2+\Gamma(2-\beta) \rho_n^{\beta}\sum_{j=2 N+1}^{n}G_{n,j}\|\bar u_h^{j-2N}\|+C\mathcal{M}^n_{\eta_\ell},
\end{align*}
for $2\ell N+1\leq n\leq2(\ell+1)N,~1\leq \ell\leq K$, where the last inequality holds by Lemma \ref{lem-G} and Lemma \ref{lem-BM}.

Now we apply the mathematical induction to obtain the following result
\[
\sqrt{\|\bar v_h^n\|^2 + \|\nabla_h\bar u_h^n\|^2}\le C t_n^{\beta} h_1^2+C\mathcal{\tilde{M}}^n_{\eta_\ell}\quad \text{for } 2\ell N+1\leq n\leq2(\ell+1)N,~1\leq \ell\leq K,
\]
where
\begin{align*}
\mathcal{\tilde{M}}^n_{\eta_\ell} :=
\begin{cases}
\sum\limits_{l=1}^{\ell}\rho(t_n-(l-1)\tau\big)^{\beta-1}& \text{if } \eta_l > 0, \\
%\le \sum\limits_{l=1}^{\ell}\rho^{1-\beta}(t_n-(l-1)\tau\big)^{\beta-1}
\sum\limits_{l=1}^{\ell}\rho(t_n-(l-1)\tau\big)^{\beta-1}\big[1 + \ln\big((t_n-(\ell-1)\tau)/\rho\big)\big] & \text{if } \eta_l = 0, \\ \sum\limits_{l=1}^{\ell}\rho^{\eta_l+1}(t_n-(\ell-1)\tau\big)^{\beta-\eta_l-1} \ln(2N)   &\text{if }\beta-1\le\eta_l<0,
  \\
\sum\limits_{l=1}^{\ell}\rho^{\eta_l+1}(t_n-(l-1)\tau\big)^{\beta-\eta_l-1} & \text{if } \eta_l <  \beta-1.
\end{cases}
\end{align*}
First, when $2N+1\le n\le 4N$ (i.e., $\ell=1$),  $\|\bar u_h^{j-2N}\|=0\text{ for }2N+1\le j\le n$, then we have
\begin{align*}
\sqrt{\|\bar v_h^n\|^2 + \|\nabla_h\bar u_h^n\|^2}
\leq C t_n^{\beta} h_1^2+C\mathcal{M}^n_{\eta_1}\le C t_n^{\beta} h_1^2+\mathcal{\tilde{M}}^n_{\eta_1}.
\end{align*}
Assume that $\sqrt{\|\bar v_h^n\|^2 + \|\nabla_h\bar u_h^n\|^2} \leq C t_n^{\beta} h_1^2+\mathcal{\tilde{M}}^n_{\eta_{l}}$ holds for $2l N+1\le n \le 2(l+1)N,~l=1,\ldots,\ell-1.$
Next, we consider the case $2\ell N+1\le n \le 2(\ell+1)N$, and we get
\begin{align}\nonumber
&\Gamma(2-\beta) \rho_n^{\beta}\sum_{j=2 N+1}^{n}G_{n,j}\|\bar u_h^{j-2N}\|\\\label{delay}
\le& C t_n^{\beta} h_1^2+C\Gamma(2-\beta) \rho_n^{\beta}\Big(\sum_{p=2}^{\ell-1}\sum_{j=2pN+1}^{2(p+1)N}G_{n,j}\mathcal{\tilde{M}}^{j-2N}_{\eta_{p-1}}+
\sum_{j=2\ell N+1}^{n}G_{n,j}\mathcal{\tilde{M}}^{j-2N}_{\eta_{\ell-1}}\Big).
\end{align}
Then, in order to estimate the second term of the right hand side in (\ref{delay}), we give that $\sum_{l=1}^{\ell}\rho(t_j-(l-1)\tau\big)^{\beta-1}\le \sum_{l=1}^{\ell}\big(\rho/(t_j-(l-1)\tau)\big)^{-\beta+1}$ when $\eta_l>0$  and $\sum_{l=1}^{\ell}\rho(t_j-(l-1)\tau\big)^{\beta-1}\big[1 + \ln\big((t_n-(\ell-1)\tau)/\rho\big)\big]\le \sum_{l=1}^{\ell}\big(\rho/(t_j-(l-1)\tau)\big)^{-\beta+1}\ln(2N)$ when $\eta_l=0$ can be integrated into the case $\beta-1\le\eta_l<0$.
And also because here $\eta_1=\cdots=\eta_\ell$ when $\eta_l< 0,~l=1,2,\ldots,\ell$, then we have
\begin{align*}
&\Gamma(2-\beta) \rho_n^{\beta}\Big(\sum_{p=2}^{\ell-1}\sum_{j=2pN+1}^{2(p+1)N}G_{n,j}\mathcal{\tilde{M}}^{j-2N}_{\eta_{p-1}}+ \sum_{j=2\ell N+1}^{n}G_{n,j}\mathcal{\tilde{M}}^{j-2N}_{\eta_{\ell-1}}\Big)\\
=& \Gamma(2-\beta) \rho_n^{\beta}\Big(\sum_{p=2}^{\ell-1}\sum_{j=2pN+1}^{2(p+1)N}G_{n,j}\mathcal{\tilde{M}}^{j}_{\eta_{p}}+ \sum_{j=2\ell N+1}^{n}G_{n,j}\mathcal{\tilde{M}}^{j}_{\eta_{\ell}}\Big)\\
\le& \Gamma(2-\beta) \rho_n^{\beta}\Big(\sum_{p=1}^{\ell-1}\sum_{j=2pN+1}^{2(p+1)N}G_{n,j}\mathcal{\tilde{M}}^{j}_{\eta_{p}}+\sum_{j=2\ell N+1}^{n}G_{n,j}\mathcal{\tilde{M}}^{j}_{\eta_{\ell}}\Big)\quad\text{for } \eta_l< 0,
\end{align*}
where the last inequality holds for $G_{n,j}>0$ and $\mathcal{\tilde{M}}^{j}_{\eta_{\ell}}>0$.
Finally, it follows from Lemmas \ref{lem-B} and Theorem \ref{theorem-M} that
\begin{align*}
\Gamma(2-\beta) \rho_n^{\beta}\Big(\sum_{p=1}^{\ell-1}\sum_{j=2pN+1}^{2(p+1)N}G_{n,j}\mathcal{\tilde{M}}^{j}_{\eta_{p}}+ \sum_{j=2\ell N+1}^{n}G_{n,j}\mathcal{\tilde{M}}^{j}_{\eta_{\ell}}\Big)\le C\mathcal{\tilde{M}}^n_{\eta_\ell} \quad\text{for } \eta_l< 0.
\end{align*}
Then, it holds that
\begin{align*} \sqrt{\|\bar v_h^n\|^2 + \|\nabla_h\bar u_h^n\|^2}
\le C t_n^{\beta} h_1^2+C\mathcal{\tilde{M}}^n_{\eta_\ell} \quad \text{for } 2\ell N+1\le n \le 2(\ell+1)N.
\end{align*}
Therefore, the claimed result is obtained.

Now we estimate \( \mathcal{\tilde{M}}^{n}_{\eta_\ell} \) by following four different cases.
\begin{enumerate}
\item  If \( 1 \leq \gamma < 2 - \beta \), then one has \( (2 - \beta)/\gamma > 1 \) and \( l\beta + 1 > 1 \). Hence, we get \( \eta_l > 0 \). Combining this with \( \rho \simeq N^{-\gamma} \), we get \( \mathcal{\tilde{M}}^n_{\eta_\ell}  \leq CN^{-\gamma}(t_n-(\ell-1)\tau)^{\beta-1} \).

\item If \( \gamma = 2 - \beta \), then we have \( (2 - \beta)/\gamma = 1 \). However, \(l\beta + 1 > 1 \), so \( \eta_l = 0 \). By using \( \rho \leq CN^{-\gamma} = CN^{\beta-2} \), we obtain \( \mathcal{\tilde{M}}^n_{\eta_\ell}  \leq CN^{\beta-2}(t_n-(\ell-1)\tau)^{\beta-1}[1 + \ln((t_n-(\ell-1)\tau)/\rho)] \).

\item If \( 2 - \beta <\gamma \le (2-\beta)/\beta \), then one has \( (2 - \beta)/\gamma < 1 \). While \(l\beta + 1 > 1 \), so \( \eta_1=\cdots=\eta_\ell < 0 \). Thus
\begin{align*}
\mathcal{\tilde{M}}^n_{\eta_\ell}  \leq &C\rho (t_n-(\ell-1)\tau)^{\beta-1}\big(\rho/(t_n-(\ell-1)\tau)\big)^{(2-\beta)/\gamma-1}\ln(2N)\\
=& C\rho^{(2-\beta)/\gamma}(t_n-(\ell-1)\tau)^{\beta-(2-\beta)/\gamma}\ln(2N)\\
\leq& CN^{\beta-2}(t_n-(\ell-1)\tau)^{\beta-(2-\beta)/\gamma}\ln(2N).
\end{align*}

\item If \(\gamma > (2-\beta)/\beta \), then one has \( (2 - \beta)/\gamma < 1 \). While \(l\beta + 1 > 1 \), so \( \eta_l < 0 \). Thus
\begin{align*} \mathcal{\tilde{M}}^n_{\eta_\ell}  \leq CN^{\beta-2}(t_n-(\ell-1)\tau)^{\beta-(2-\beta)/\gamma}.     \end{align*}
\end{enumerate}
Hence, the above four cases show that \( \mathcal{\tilde{M}}^n_{\eta_\ell} \leq C\mathbb{M}^n \) holds for $2\ell N+1\leq n\leq2(\ell+1)N,~1\leq \ell\leq K$. Finally, our proof is completed.
\end{proof}

\begin{corollary}[\text{Local error}]
Suppose \( t_n \geq t_*>(\ell-1)\tau \), where \( t_* \) represents a fixed positive constant and  $(\ell-1)\tau< t_n\leq \ell\tau,~1\leq \ell\leq K$. If the assumption given in Theorem \ref{thm-error} holds, one has
\[
\max_{t_* \leq t_n \leq \ell\tau} \|\bar v^n\| + \max_{t_* \leq t_n \leq \ell\tau} \|\nabla_h \bar u_h^n\| \leq
\begin{cases}
C\big(h_1^2  + N^{-\gamma}\big) & \text{if } 1\le\gamma < 2 - \beta, \\
C\big(h_1^2 +  N^{-(2-\beta)}\ln (2N)\big) & \text{if } 2 - \beta\le\gamma \le (2-\beta)/\beta, \\  C\big(h_1^2 +  N^{-(2-\beta)}\big) & \text{if } \gamma > (2-\beta)/\beta.
\end{cases}
\]
%for $2\ell N+1\leq n\leq2(\ell+1)N,~1\leq \ell\leq K$.
\end{corollary}

\begin{proof}
If \( 1 \leq \gamma<2 - \beta \), one has \( \max_{t_* \leq t_n \leq \ell\tau} \mathbb{M}^n \leq CN^{-\gamma} \). For \( \gamma > (2-\beta)/\beta \), we have \( \max_{t_* \leq t_n \leq \ell\tau} \mathbb{M}^n \leq CN^{-(2-\beta)} \).
If \( 2- \beta\le\gamma \le (2-\beta)/\beta \), then we have \( \max_{t_* \leq t_n \leq \ell\tau} \mathbb{M}^n \leq CN^{-(2-\beta)} \ln (2N )\). Hence, combining these results with (\ref{th-error}) of Theorem \ref{thm-error}, our proof is finished.
\end{proof}

%\subsubsection{The local error analysis of SYM $L1$-ADI scheme (\ref{num-scheme-4})}

\section{Numerical experiments}\label{sec-num}
In this section, we carry out numerical experiments to illustrate our theoretical statements and all our tests are done in MATLAB with a laptop. Let $M_1=M_2=M$ for two-dimensional problem. Let
$$E_{H^1}(M,N)=\max_{2N+1\leq n\leq 2(K+1)N}\|u_h^n-U_h^n\|_{H^1},~E(M,N)=\|u_h^n-U_h^n\|_{H^1}$$ where $u_h^n$ and $U_h^n$ are the reference solution and the numerical solution, respectively.
Furthermore, the temporal convergence order and spatial convergence order, denoted by
$$\mbox{Order1}=\log_2\bigg(\frac{E_{H^1}(M,N/2)}{E_{H^1}(M,N)}\bigg)
 ~~\mbox{and}~~ \mbox{Order2}=\log_2\bigg(\frac{E_{H^1}(M/2,N)}{E_{H^1}(M,N)}\bigg),$$
respectively, are reported. Meanwhile, the local convergence order in temporal direction is denoted by $$\mbox{Rate}=\log_2\bigg(\frac{E(M,N/2)}{E(M,N)}\bigg).$$

\iffalse
Let us revisit the foundational work in which the SFOR method was originally proposed \cite{LyuP2022SFOR}. As established in Remark 2.2 of that paper, auxiliary variables were  precisely introduced to avoid explicitly representing the singular temporal kernel $a_1(x)\omega_{2-\alpha}(t)$. Based on this, they adopted the following numerical framework:
\begin{equation}\label{eq-gov-ref}
\begin{cases}
  \partial_t^\beta \textbf{v}  - \Delta \textbf{u} = t\Delta a_1(x)+g(x,t), & (x,t) \in \Omega \times (0,T), \\
  \textbf{v}  = \partial_t^\beta \textbf{u}, & (x,t) \in \Omega \times (0,T), \\
  \textbf{u}(x,0) = \textbf{v}(x,0)=0, & x \in \Omega, \\
  \textbf{u}(x,t) = \textbf{v}(x,t)=0, & (x,t) \in \partial \Omega \times (0,T),
\end{cases}
\end{equation}
where $u=\textbf{u}+ta_1(x),~\frac12<\beta<1$. The limitation of (\ref{eq-gov-ref}) is $(\Delta a_1(x),\phi(x))$, $x\in\Omega$ must exist, where $\phi(x)$ is a basis function from a finite element space.\fi

The $L1$ and SFOR methods are employed to simulate the problems (\ref{eq-gov}). To rigorously validate the theoretical properties of the proposed schemes, we conduct two carefully designed numerical experiments. In our numerical framework (\ref{eq-gov-trans}), our formula does not contain any redundant singular kernel terms by introducing the auxiliary variable $v = \partial_t^{\frac{\alpha}{2}} u$. First, we find that the optimal convergence order of $L1$ scheme is up to $2-\alpha/2$, which improves upon the classical rate of $2-\alpha$ under identical smoothness conditions, when the mesh parameter of graded meshes is $\gamma=\frac{4-\alpha}{\alpha}$ for $\epsilon<\alpha<1$, where $\epsilon$ is a fixed positive constant. Theoretically, the scheme becomes inapplicable as $\alpha\rightarrow0^+$ because the mesh parameter $\gamma=\frac{4-\alpha}{\alpha}$ diverges to $\infty$, rendering the underlying functional framework ill-defined. But we can yield optimal convergence rates away from singular time points under milder grading parameter $\gamma=2-\alpha/2$,  which is particularly well-suited for cases with small $\alpha$. And we can observe from Table \ref{table1} that the SYM $L1$ scheme (\ref{num-scheme-1}) at $t=1$ achieves the optimal convergence order of $2-\alpha/2$ even when \(\alpha = 0.01\), thereby compensating for the shortcomings of the global convergence analysis.

Next, we examine the temporal convergence behavior of both methods under varying fractional order $\alpha \in (0,1)$ and mesh grading parameter $\gamma\geq 1$, with results summarized in Tables \ref{table-test1}-\ref{table-test7}. First, Tables \ref{table-test1}-\ref{table-test3} confirm that the SYM $L1$ scheme (\ref{num-scheme-1}) achieves the predicted convergence rate of $2-\alpha/2$ when $\gamma = \frac{4-\alpha}{\alpha}$—a value derived from our error analysis to compensate for solution singularity. In contrast, under uniform temporal discretization ($\gamma=1$), the observed convergence order drops below one, corroborating the necessity of graded meshes for optimal accuracy. Second, we can observe from Table \ref{table-test4} that the optimal convergence order of the SYM $L1$ scheme (\ref{num-scheme-1}) is $2-\alpha/2$, while that of the conventional $L1$ scheme is $2-\alpha$.
This clearly demonstrates the superior performance of SYM $L1$ scheme (\ref{num-scheme-1}) is better than that of the standard $L1$ scheme. Third, we apply the SYM $L1$-ADI scheme (\ref{num-scheme-4}) to system (\ref{eq-gov-trans}) using its theoretically optimal mesh parameters with different parameter $\alpha$; the resulting errors, reported in Table \ref{table-test6}, align precisely with the sharp convergence rates established in our error bounds. Finally, the spatial convergence orders of SYM $L1$-ADI scheme (\ref{num-scheme-4}) and SYM $L1$-ADI compact scheme (\ref{COM-ADI1})-(\ref{COM-ADI2}) which is a spatial fourth-order accurate numerical scheme, presented in Table \ref{table-test7}, further validate the robustness of our theoretical convergence results.

\iffalse
\begin{remark}\label{ill-ref}
The point is \textbf{v} become more regular. Following the idea of Lemma \ref{thm-reg-L2-v}, one has
$\|\partial_t^m\textbf{v}\|\leq Ct^{1+\beta-m}\|a_1\|$, $m=0,1,2,3$. Combining the analysis in Lemma \ref{thm-conv-1}, it implies that the mesh parameter $r=2$ is enough.
\end{remark}
\fi

\begin{example}\label{ex-test}
Problems {\rm(\ref{eq-gov})} with $\Omega=(0,\pi)$, $\tau=0.5$, $K=2$,
$a_0(x,t)=\sin(x)\exp(t)$ and $f(x,t)=\sin(x)t^2$.
 We use the numerical solution with the size of the space grids $h_1=\frac{\pi}{100}$ and $N=400$ being the number of partitions in the time grids as the reference solution for SYM $L1$ scheme {\rm(\ref{num-scheme-1})}.
\end{example}

\begin{table}[!ht]
\caption{Temporal accuracy at $t_n=1$ of SYM $L1$ scheme (\ref{num-scheme-1}) for Example \ref{ex-test} with $\alpha=0.01$.}\label{table1}  \renewcommand{\arraystretch}{1}
\def\temptablewidth{0.95\textwidth}
\begin{center}
\begin{tabular*}
{\temptablewidth}{@{\extracolsep{\fill}}lcccccc}\hline
$N$ &\multicolumn{2}{c}{$\gamma=1$} &\multicolumn{2}{c}{$\gamma=2-\alpha/2$} &\multicolumn{2}{c}{$\gamma=2$}    \\
\cline{2-3}\cline{4-5}\cline{6-7}
&$E(M,N)$          &Rate     &$E(M,N)$    &Rate    &$E(M,N)$  &Rate     \\ \hline
10 &3.9828e-05    &-  &1.5117e-06   &-   &1.4530e-06  &-    \\
20 &1.9805e-05    &1.0079  &3.8334e-07   &1.9795   &3.6489e-07  &1.9935  \\
40 &9.4982e-06    &1.0602  &9.7634e-08   &1.9731   &9.2101e-08  &1.9862    \\
80 &4.2519e-06    &1.1596  &2.4336e-08   &2.0043   &2.2757e-08  &2.0169    \\\hline
Optimal Order  &\multicolumn{6}{c}{1.995} \\ \hline
\end{tabular*}
\end{center}
\end{table}

\begin{table}[!ht]
\caption{Temporal accuracy of
SYM $L1$ scheme (\ref{num-scheme-1}) for Example \ref{ex-test} with $\alpha=0.4$.}\label{table-test1}
\renewcommand{\arraystretch}{1}
\def\temptablewidth{0.95\textwidth}
\begin{center}
\begin{tabular*}{\temptablewidth}{@{\extracolsep{\fill}}lcccccc}\hline
$N$ &\multicolumn{2}{c}{$\gamma=1$} &\multicolumn{2}{c}{$\gamma=\frac{2-\alpha/2}{\alpha}$} &\multicolumn{2}{c}{$\gamma=5$}    \\
\cline{2-3}\cline{4-5}\cline{6-7}
& $E_{H^1}(M,N)$         &Order1     &$E_{H^1}(M,N)$    &Order1    &$E_{H^1}(M,N)$ &Order1     \\ \hline
10 &7.3791e-02    &-  &7.5740e-03   &-   &7.0186e-03  &-    \\
20 &5.9422e-02    &0.3125  &2.4709e-03   &1.6160   &2.2424e-03  &1.6461  \\
40 &4.5836e-02    &0.3745  &7.7051e-04   &1.6812   &6.9502e-04  &1.6899    \\
80 &3.2802e-02    &0.4827  &2.2941e-04   &1.7479   &2.0668e-04  &1.7497    \\\hline
Optimal Order  &\multicolumn{6}{c}{1.8} \\ \hline
\end{tabular*}
\end{center}
\end{table}

\begin{table}[!ht]
\caption{Temporal accuracy of
SYM $L1$ scheme (\ref{num-scheme-1}) for Example \ref{ex-test} with $\alpha=0.6$.}\label{table-test2}
\renewcommand{\arraystretch}{1}
\def\temptablewidth{0.95\textwidth}
\begin{center}
 \begin{tabular*}{\temptablewidth}{@{\extracolsep{\fill}}lcccccc}\hline
 $N$ &\multicolumn{2}{c}{$\gamma=1$} &\multicolumn{2}{c}{$\gamma=\frac{2-\alpha/2}{\alpha}$} &\multicolumn{2}{c}{$\gamma=\frac{4-\alpha}{\alpha}$}    \\
 \cline{2-3}\cline{4-5}\cline{6-7}
 & $E_{H^1}(M,N)$         &Order1     &$E_{H^1}(M,N)$    &Order1   &$E_{H^1}(M,N)$ &Order1     \\ \hline
 10 &5.6678e-02    &-  &9.1953e-03   &-   &1.4380e-02  &-    \\
 20 &3.8046e-02    &0.5751  &3.1492e-03   &1.5459   &4.9337e-03  &1.5433  \\
 40 &2.4315e-02    &0.6459  &1.0198e-03   &1.6267   &1.6325e-03  &1.5955    \\
 80 &1.4234e-02    &0.7725  &3.1268e-04   &1.7056   &5.1415e-04  &1.6669    \\\hline
 Optimal Order  &\multicolumn{6}{c}{1.7} \\ \hline
\end{tabular*}
\end{center}
\end{table}

\begin{table}[!ht]
\caption{Temporal accuracy of
SYM $L1$ scheme (\ref{num-scheme-1}) for Example \ref{ex-test} with $\alpha=0.8$.}\label{table-test3}
\renewcommand{\arraystretch}{1}
\def\temptablewidth{0.95\textwidth}
\begin{center}
\begin{tabular*}
{\temptablewidth}{@{\extracolsep{\fill}}lcccccc}\hline
$N$ &\multicolumn{2}{c}{$\gamma=1$} &\multicolumn{2}{c}{$\gamma=\frac{2-\alpha/2}{\alpha}$} &\multicolumn{2}{c}{$\gamma=\frac{4-\alpha}{\alpha}$}    \\
\cline{2-3}\cline{4-5}\cline{6-7}
& $E_{H^1}(M,N)$         &Order1     &$E_{H^1}(M,N)$    &Order1    &$E_{H^1}(M,N)$ &Order1     \\ \hline
10 &3.2927e-02    &-  &1.1184e-02   &-   &1.7423e-02  &-    \\
20 &1.9054e-02    &0.7892  &4.0923e-03   &1.4504   &6.3047e-03  &1.4665  \\
40 &1.0482e-02    &0.8623  &1.4350e-03   &1.5118   &2.1563e-03  &1.5479    \\
80 &5.2938e-03    &0.9855  &4.7405e-04   &1.5980   &6.9759e-04  &1.6281    \\\hline
Optimal Order  &\multicolumn{6}{c}{1.6} \\ \hline
\end{tabular*}
\end{center}
\end{table}

\begin{table}[!ht]
\begin{center}
\caption{Convergence orders in temporal direction for Example \ref{ex-test} with $\alpha=0.99$.}\label{table-test4}
\renewcommand{\arraystretch}{1.0}
\def\temptablewidth{0.9\textwidth}
{\rule{\temptablewidth}{0.7pt}}
 \begin{tabular*}{\temptablewidth}{@{\extracolsep{\fill}}lcccccccc}
 $\gamma$ &$N$  &\multicolumn{2}{c}{SYM $L1$ Scheme (\ref{num-scheme-1})} &\multicolumn{2}{c}{$L1$ Scheme \cite{Delay01-Tan}} \\
 \cline{3-4}\cline{5-6}
    &        & $E_{H^1}(M,N)$     &Order1     &$E_{H^1}(M,N)$    &Order1  \\\hline
 $\frac{2-\alpha}{\alpha}$
    &10     & 2.2447e-02      & -     & 3.5700e-02    & -  \\
    &20     & 1.1532e-02      & 0.9609     & 1.7450e-02    & 1.0327  \\
    &40     & 5.7132e-03      & 1.0133     & 8.2536e-03    & 1.0801 \\
    &80     & 2.6287e-03      & 1.1200     & 3.6552e-03    & 1.1751
    \\\hline
 $\frac{2-\alpha/2}{\alpha}$
    &10     & 1.4795e-02     & -      & 4.2783e-02    & -  \\
    &20     & 5.8491e-03     & 1.3388      & 2.1018e-02    & 1.0254  \\
    &40     & 2.2384e-03     & 1.3858      & 9.9678e-03    & 1.0763 \\
    &80     & 8.1483e-04     & 1.4579      & 4.4212e-03    & 1.1728
     \\\hline
  \multicolumn{2}{c}{Optimal Order}
 &\multicolumn{2}{c}{1.505}&\multicolumn{2}{c}{1.01}
 \end{tabular*}
 {\rule{\temptablewidth}{0.7pt}}
 \end{center}
 \end{table}

\iffalse
\begin{table}[!ht]
\caption{Temporal accuracy at $t=1$ of SYM $L1$ scheme (\ref{num-scheme-1}) for Example \ref{ex-test} with $\alpha=0.01$.}\label{table1}  \renewcommand{\arraystretch}{1}
\def\temptablewidth{0.95\textwidth}
\begin{center}
\begin{tabular*}
{\temptablewidth}{@{\extracolsep{\fill}}lcccccc}\hline
$N$ &\multicolumn{2}{c}{$\gamma=1$} &\multicolumn{2}{c}{$\gamma=2-\alpha/2$} &\multicolumn{2}{c}{$\gamma=2$}    \\
\cline{2-3}\cline{4-5}\cline{6-7}
&$E(M,N)$          &Rate     &$E(M,N)$    &Rate    &$E(M,N)$  &Rate     \\ \hline
5  &7.8177e-05    &-       &6.1966e-06   &-        &6.0220e-06  &-  \\
10 &3.9828e-05    &0.9730  &1.5117e-06   &2.0353   &1.4530e-06  &2.0512    \\
20 &1.9805e-05    &1.0079  &3.8334e-07   &1.9795   &3.6489e-07  &1.9935  \\
40 &9.4982e-06    &1.0602  &9.7634e-08   &1.9731   &9.2101e-08  &1.9862    \\
80 &4.2519e-06    &1.1596  &2.4336e-08   &2.0043   &2.2757e-08  &2.0169    \\\hline
Optimal Order  &\multicolumn{6}{c}{1.995} \\ \hline
\end{tabular*}
\end{center}
\end{table}
\fi

\begin{example}\label{ex2}
Problems {\rm(\ref{eq-gov})} with $\Omega=(0,\pi)\times(0,\pi)$, $\tau=0.25$, $K=2$,
$a_0(x,y,t)=\sin(x)\sin(y)\exp(t)$ and $f(x,y,t)=\sin(x)\sin(y)\cos(t)$.
We use the numerical solution with the size of the space grids $h_1=h_2=\frac{\pi}{100}$ and $N=400$ being the number of partitions in the time grids as the reference solution for SYM $L1$-ADI scheme {\rm(\ref{num-scheme-4})}. Meanwhile, the numerical solution with $N=100$ and $M_1=M_2=400$ is the reference solution for SYM $L1$-ADI compact scheme {\rm(\ref{COM-ADI1})-(\ref{COM-ADI2})}.
\end{example}

\iffalse
\begin{table}[!ht]
\caption{Temporal accuracy of
SYM $L1$-ADI scheme (\ref{num-scheme-4}) for Example \ref{ex2} with $r=1$.}\label{table-test5}
\renewcommand{\arraystretch}{1}
\def\temptablewidth{0.95\textwidth}
\begin{center}
 \begin{tabular*}{\temptablewidth}{@{\extracolsep{\fill}}lcccccc}\hline
 $N$ &\multicolumn{2}{c}{$\alpha=0.4$} &\multicolumn{2}{c}{$\alpha=0.6$} &\multicolumn{2}{c}{$\alpha=0.8$}    \\
 \cline{2-3}\cline{4-5}\cline{6-7}
 & $E_{H^1}(M,N)$         &Order1     &$E_{H^1}(M,N)$    &Order1    &$E_{H^1}(M,N)$ &Order1     \\ \hline
 10 &4.1365e-01     &-        &3.4786e-01   &-        &1.8639e-01  &-  \\
 20 &3.6118e-01     &0.1957   &2.4262e-01   &0.5198   &1.1052e-01  &0.7541    \\
 40 &2.9812e-01     &0.2768   &1.5928e-01   &0.6071   &6.1666e-02  &0.8417 \\
 80 &2.2562e-01     &0.4020   &9.4997e-02   &0.7456   &3.1410e-02  &0.9733
\\ \hline
\end{tabular*}
\end{center}
\end{table}
\fi

\begin{table}[!ht]
\caption{Temporal accuracy of
SYM $L1$-ADI scheme (\ref{num-scheme-4}) for Example \ref{ex2} with $\gamma=3/2$.}\label{table-test6}
\renewcommand{\arraystretch}{1}
\def\temptablewidth{0.95\textwidth}
\begin{center}
 \begin{tabular*}{\temptablewidth}{@{\extracolsep{\fill}}lcccccc}\hline
 $N$ &\multicolumn{2}{c}{$\alpha=0.4$} &\multicolumn{2}{c}{$\alpha=0.6$} &\multicolumn{2}{c}{$\alpha=0.8$}    \\
 \cline{2-3}\cline{4-5}\cline{6-7}
 & $E_{H^1}(M,N)$         &Order1     &$E_{H^1}(M,N)$    &Order1    &$E_{H^1}(M,N)$ &Order1     \\ \hline
 10 &3.3785e-01     &-        &2.0099e-01   &-        &9.5773e-02  &-  \\
 20 &2.5784e-01     &0.3899   &1.1273e-01   &0.8343   &4.4486e-02  &1.1063    \\
 40 &1.8435e-01     &0.4841   &6.0033e-02   &0.9090   &1.9478e-02  &1.1915 \\
 80 &1.2099e-01     &0.6076   &2.9415e-02   &1.0292   &7.9851e-03  &1.2865
     \\\hline
Optimal Order &\multicolumn{2}{c}{0.6}  &\multicolumn{2}{c}{0.9}&\multicolumn{2}{c}{1.2}\\ \hline
\end{tabular*}
\end{center}
\end{table}

 \begin{table}[!htbp]
\begin{center}
\caption{Spacial accuracy of ADI and compact ADI schemes  for Example \ref{ex2} with $\alpha=0.6$ and $\gamma=3$.}
 \label{table-test7}
\renewcommand{\arraystretch}{0.95}
\def\temptablewidth{0.8\textwidth}
{\rule{\temptablewidth}{0.7pt}}
 \begin{tabular*}{\temptablewidth}{@{\extracolsep{\fill}}lcccc}
 $M$ &\multicolumn{2}{c}{ADI} &\multicolumn{2}{c}{Compact ADI}\\
 \cline{2-3}\cline{4-5}
 & $E_{H^1}(M,N)$    &Order2 & $E_{H^1}(M,N)$    &Order2\\\hline
 10  & 4.5207e-03    & -     & 2.2400e-05    & -\\
 20  & 1.1784e-03    & 1.9397 & 1.4590e-06    & 3.9405\\
 40  & 2.9555e-04    & 1.9954 & 9.2396e-08    & 3.9810\\
 80  & 7.1841e-05    & 2.0405 & 5.7597e-09    & 4.0038\\ \hline
 Theoretical Order  &\multicolumn{2}{c}{2} &\multicolumn{2}{c}{4}
 \end{tabular*}
 {\rule{\temptablewidth}{0.7pt}}
 \end{center}
 \end{table}

\iffalse
\begin{table}[!ht]   \def\temptablewidth{0.95\textwidth}   \begin{center}
\caption{Temporal accuracy at $t_{6N}=1$ of SYM $L1$-ADI scheme (\ref{num-scheme-4}) for Example \ref{ex2} with $\gamma=3/2$.}\label{table1}
\begin{tabular*}  {\temptablewidth}{@{\extracolsep{\fill}}lcccccc}\hline    $N$ &\multicolumn{2}{c}{$\alpha=0.4$} &\multicolumn{2}{c}{$\alpha=0.6$} &\multicolumn{2}{c}{$\alpha=0.8$}    \\
\cline{2-3}\cline{4-5}\cline{6-7}
&$E(M,N)$    &Rate     &$E(M,N)$    &Rate    &$E(M,N)$  &Rate   \\\hline
10  &7.8177e-05    &-       &6.1966e-06   &-        &6.0220e-06  &-  \\
20 &3.9828e-05    &0.9730  &1.5117e-06   &2.0353   &1.4530e-06  &2.0512    \\
40 &1.9805e-05    &1.0079  &3.8334e-07   &1.9795   &3.6489e-07  &1.9935  \\
80 &9.4982e-06    &1.0602  &9.7634e-08   &1.9731   &9.2101e-08  &1.9862    \\\hline
Optimal Order  &\multicolumn{6}{c}{} \\ \hline
\end{tabular*}
\end{center}
\end{table}
\fi

\newpage
\section{Conclusions}\label{conclusion}
In summary, this study develops a symmetric fractional-order reduction technique for constructing efficient numerical solvers for fractional sub-diffusion equations with time delay. When coupled with the $L1$ discretization on nonuniform temporal meshes, the resulting schemes achieve a temporal convergence rate of $2-\alpha/2$, which is notably higher than the conventional $2-\alpha$ order under comparable regularity assumptions. To further enhance computational scalability in two-dimensional case, we incorporate a weighted ADI approach. The spatial compact scheme combined with the ADI method is also discussed. What's more,  we introduce a new framework for analyzing the error of $L1$-type discretization on graded temporal meshes, which relaxes the mesh grading requirement and still guarantees optimal convergence rate at continuous time points. Extensive numerical tests validate the expected convergence behavior and highlight the superior computational efficiency of the proposed algorithms relative to existing approaches.

Overall, the SFOR-based framework offers a theoretically rigorous, numerically robust, and practically efficient methodology for simulating delayed fractional diffusion problems with multiple singularities. Future work will extend the SFOR framework to variable-order fractional delayed models and distributed-delay settings; additionally, a rigorous convergence analysis of SFOR for higher-dimensional problems under relaxed regularity on the solution remains an important open direction.

\section*{Funding}
S. Vong was partially supported by University of Macau (File No. MYRG-GRG2025-00077-FST, MYRG-GRG2024-00100-FST-UMDF). Z. Wang was partially supported by the Natural Science Foundation of Guangdong Province (No. 2023A1515011504).

\section*{Data Availability}
The data that support the findings of this study are available from the corresponding author upon reasonable request

\section*{Declarations}
On behalf of all authors, the corresponding author states that there is no conflict of interest. No datasets were generated or analyzed during the current study.

\bibliographystyle{plain}

\end{document}